\documentclass[11pt, a4paper]{article}%
\usepackage{amsmath}%
\usepackage{amsfonts}%
\usepackage{amssymb}%
\usepackage{graphicx}
\usepackage{color}
\usepackage{xypic}
\usepackage{hyperref}
\hypersetup{colorlinks=red,citecolor=red,urlcolor=red,linkcolor=red}

\newtheorem{theorem}{Theorem}

\newtheorem{definition}[theorem]{Definition}

\newtheorem{lemma}[theorem]{Lemma}

\newtheorem{proposition}[theorem]{Proposition}
\newtheorem{remark}[theorem]{Remark}

\newenvironment{proof}[1][Proof]{\noindent\textbf{#1.} }{\ \rule{0.5em}{0.5em}}

\newcommand{\NN}{\mathbb N}
\newcommand{\C}{\mathbb C}

\def\C{\mathbb{C}}
\def\N{\mathbb{N}}

\def\fJ{\mathfrak{J}}
\def\fP{\mathfrak{P}}

\def\square{\ \rule{0.5em}{0.5em}}

\begin{document}

\title{On the Index and the Order of Quasi-regular Implicit Systems of Differential Equations}

\author{Lisi D'Alfonso\thanks{Partially supported by  UBACyT X112 (2004-2007) and X211 (2008-2010).} \\[0.2cm]
{\textit{\normalsize Departamento de Matem\'atica, Facultad de Ciencias
Exactas y Naturales,}}\\
{\normalsize{and} \textit{\normalsize Departamento de Ciencias Exactas, Ciclo B\'asico Com\'un,}}\\
{\textit{\normalsize Universidad de Buenos Aires, Ciudad Universitaria, 1428
Buenos Aires, Argentina}}
\\[0.3cm]\and
Gabriela Jeronimo$^*$\thanks{Partially supported by CONICET PIP
5852/05, UBACyT X847 (2006-2009) and ANPCyT PICT 17-33018/05.},
Gustavo Massaccesi\thanks{Partially supported by UBACyT X058,
CONICET PIP 5056, ANPCyT PICT 03-15033 and Fundaci\'on Ciencias
Exactas y Naturales (Argentina).}, Pablo Solern\'o$^*$
\\[0.3cm]
{\textit{\normalsize Departamento de Matem\'atica, Facultad de Ciencias
Exactas y Naturales}}\\
{\textit{\normalsize Universidad de Buenos Aires, Ciudad Universitaria, 1428
Buenos Aires, Argentina}}\\{\normalsize and  CONICET,  \textit{Argentina}}
\\[0.2cm]
{\normalsize E-mails: lisi@dm.uba.ar, jeronimo@dm.uba.ar,
gustavo@oma.org.ar, psolerno@dm.uba.ar}}


\maketitle

\begin{abstract}
This paper is mainly devoted to the study of the differentiation
index and the order for quasi-regular implicit ordinary differential
algebraic equation (DAE) systems. We give an algebraic definition of
the differentiation index and prove a Jacobi-type upper bound for
the sum of the order and the differentiation index. Our techniques
also enable us to obtain an alternative proof of a combinatorial
bound proposed by Jacobi for the order.

As a consequence of our approach we deduce an upper bound for the
Hilbert-Kolchin regularity and an effective ideal membership test
for quasi-regular implicit systems. Finally, we prove a theorem of
existence and uniqueness of solutions for implicit differential
systems.
\end{abstract}

\bigskip


\section{Introduction}

Throughout the paper we consider implicit differential algebraic
equation systems (DAE systems for short) of the following type:

\begin{equation} \label{sistema intro}
(\Sigma):=\left\{
\begin{array}
[c]{ccl}%
f_1(X_1,\dot{X_1},\ldots, X_1^{(\epsilon_{11})},\ldots, X_n, \dot X_n, \ldots, X_n^{(\epsilon_{1n})}) &=& 0 \\
&\vdots &\\
f_r(X_1,\dot{X_1},\ldots,X_1^{(\epsilon_{r1})},\ldots, X_n, \dot X_n, \ldots, X_n^{(\epsilon_{rn})}) &=& 0 \\
\end{array}
\right.
\end{equation}
where $r\le n$, and $f_1,\dots, f_r$ are polynomials in the $n$
differential unknowns $X:=X_1,\ldots, X_n$ and some of their
derivatives with coefficients in a characteristic zero differential
field $K$ (i.e. a  field with a derivation, for instance
$K:=\mathbb{Q}, \mathbb{R}$ or $\mathbb{C}$ with the null
derivation, or a field of rational functions in one variable with
the standard derivation). Each non-negative integer $\epsilon_
{ij}$, $1\le i\le r$, $ 1\le j\le n$,  denotes the maximal order of
derivation of the variable $X_j$ appearing in the polynomial $f_i$.

The set of all total formal derivatives of the equations defining
$(\Sigma)$ generates an ideal $[F]$, closed by differentiation, in
the (algebraically non-noetherian) polynomial ring $K\{X\}$ in the
infinitely many variables $X_j^{(l)}$, $j=1,\ldots,n$, $l\in
\mathbb{N}_0$. Since this ideal is not necessarily prime, we also
fix a \emph{minimal} prime ideal $\frak{P}\subset K\{X\}$ containing $[F]$.

We require a regularity assumption on the system $(\Sigma)$
with respect to the prime ideal $\fP$: we suppose that the system
itself, as well as the systems obtained by successive total
derivations of the defining equations, are non-singular complete
intersection algebraic varieties at the (not necessarily closed)
point $\frak{P}$. In terms of K\"ahler differentials this condition
is formalized saying that the differentials $\{
\textrm{d}f_i^{(k)},\ 1\le i\le r,\ k\in \mathbb{N}_0\} \subset
\Omega_{K\{X\}|K}$ are a $K\{X\}/\fP -$linearly independent set in
$\Omega_{K\{X\}|K}\otimes_{K\{X\}} K\{X\}/\fP $. A DAE system with this property is
called \emph{quasi-regular at $\fP$} (see, for instance,
\cite{kondra}, \cite{ollisad}). In addition, we make use of
a rather technical
hypothesis which holds for a wide class of DAE systems
(see Subsection \ref{seccionjacobiana} below)

The main invariant we consider in this paper is the
\emph{differentiation index of the system $(\Sigma)$}. There are
several definitions of differentiation indices not all completely
equivalent (see for instance \cite{brenan}, \cite{campbell-gear},
\cite{Fliess}, \cite{kunkel}, \cite{levey}, \cite{pantelides},
\cite{poulsen}, \cite{sit}, \cite{rabier}, \cite{reid},
\cite{seiler}), but in every case it represents a measure of the
implicitness of the given system. For instance, for first order
systems, differentiation indices provide bounds for the number of
total derivatives of the system needed in order to obtain an
equivalent ODE system (see \cite[Definition 2.2.2]{brenan}). Thus,
differentiation indices turn out to be closely related to the
complexity of the traditional numerical methods applied to solve
these systems (see \cite[Theorem 5.4.1]{brenan}). Since explicitness
is strongly related to the existence of classical solutions, a
differentiation index should also bound the number of derivatives
needed in order to obtain existence and uniqueness theorems (see
\cite{pritchard}, \cite{sit}, \cite{rabier}).

Following the construction given in our previous paper \cite{DJS07},
we introduce here a new differentiation index for a DAE system
$(\Sigma)$ quasi-regular at a minimal prime ideal $\fP$, that we
call \emph{the $\fP$-differentiation index}. As usual, its
definition follows from a certain chain which eventually becomes
stationary. In our case, this chain condition is simply established
by the sequence of ranks of certain Jacobian submatrices associated
to the input equations and their total derivatives (see Theorem
\ref{definicion de indice} below). This approach enables us to show
in an easy way several properties and consequences of our
differentiation index.

In particular, we show that the $\fP$-differentiation index is
closely related to the number of derivatives of the system needed to
obtain the manifold of all constraints that must be satisfied by its
solutions. More precisely, we prove that for every order $h$, all
the differential consequences of order $h$ of the quasi-regular
system $(\Sigma)$ can be obtained from the first $h+\sigma
-\max\{\epsilon_{ij}\}$ derivatives of the equations (see Theorem
\ref{sadik}).

The second invariant of the system $(\Sigma)$ we consider is the
\emph{order} of the prime differential ideal $\fP$ (which will also
be called the order of $(\Sigma)$ in the case where the differential
ideal generated by the equations is prime). Roughly speaking, the
\emph{order} of a DAE system is the number of initial conditions
that can be prefixed arbitrarily.

In his posthumous papers \cite{jacobi} and \cite{jacobi2}, Jacobi
introduced the parameter $J(\mathcal{E}):=\max\{\sum_{i=1}^r
\epsilon_{i\tau(i)} \mid \tau:\{1,\dots, r\} \to \{1,\dots, n\}
\text{ is an injection}\}$ as an optimal estimation for the
\emph{order} of an ordinary DAE system (here $\epsilon_{ij}:=-\infty$ if
the variable $X_j$ does not appear in $f_i$).
This result should be
considered as a conjecture, since Jacobi's proof seems
not to be complete in the general case. Nevertheless,
in the zero-dimensional case, it
has been proved for
linear systems by Ritt \cite[Ch.~VII, p.~135]{ritt}
and extended to quasi-regular DAE systems
in \cite{kondra} (see also \cite{ollisad}).
Other less accurate upper bounds for the order, mainly of B\'ezout
type, have been given for general (zero-dimensional) systems (see,
for instance, \cite{cohn}, \cite{DJS07}, \cite[Ch.~IV,
Prop.~9]{kol}, \cite{ritt}, \cite{sadikt}).

Here, we establish an \emph{a priori} upper bound for the sum of the
$\fP$-differentiation index and the order of $\fP$ in the
quasi-regular setting, in terms of the maximal derivation orders
$\epsilon_{ij}$ of the variables in the input equations. More
precisely, we prove the following inequality:
\[\sigma +{\rm ord}(\fP) \le J(\mathcal{E}_0) + \max\{\epsilon_{ij}\} -
\min\{\epsilon_{ij}\},\] where  $\sigma$ denotes the
$\fP$-differentiation index of the system $(\Sigma)$, ${\rm
ord}(\fP)$ the order of the differential ideal $\fP$ and $J(\mathcal{E}_0)$ the
\emph{weak} Jacobi number $J(\mathcal{E}_0):=\max\{\sum_{i=1}^r
\epsilon_{i\tau(i)} \mid \tau:\{1,\dots, r\} \to \{1,\dots, n\}
\text{ is an injection}\}$ with $\epsilon_{ij}:=0$ if
the variable $X_j$ does not appear in $f_i$. In particular,
we deduce an upper bound for the differentiation index in terms of
the weak Jacobi number  $J(\mathcal{E}_0)$ (see \cite{DJS07},
\cite{pritchard}, \cite{rabier} for less precise estimations).

Moreover, our techniques enable us to recover the
original Jacobi bound for the order
of quasi-regular DAE systems (see Theorem
\ref{cotaorden}). Our strategy may be considered as a generalization
of Lando's method introduced to obtain a weak Jacobi bound for the
order of linear differential equations (see \cite{lando}).

Our approach, which is reminiscent of the classical
\emph{completion} techniques in partial differential algebraic
equations (see \cite{cartan}, \cite{janet}, \cite{kuranishi} and
\cite{riquier}), also enables us to exhibit an upper bound for the
\emph{Hilbert-Kolchin regularity} of the ideal $\fP$ depending only
on the maximal order of derivation of the variables involved in the
system (see Theorem \ref{hilbert constante}). A precise estimation
of this regularity can be obtained if a characteristic set of
the ideal $\fP$ is known
(see \cite[Ch.II, Section 12, Th.6 (d)]{kol}) while in our case this
requirement is not necessary.

As a further consequence of the previous results, we deduce  an
effective differential ideal membership test for quasi-regular DAE
systems (Theorem \ref{member bound}): we obtain asymptotically
optimal bounds on the order and degrees for the representation of a
differential polynomial lying in the ideal in terms of the given
generators (see \cite{Seidenberg}, \cite{Boulier} and
\cite{golubitsky} for related works).

Finally, in the zero-dimensional case ($n=r$), we obtain a result
concerning the number of derivatives of the input equations required
to obtain an explicit ODE system (see Theorem \ref{despeje general
dimension 0}) and an existence and uniqueness theorem (see Theorem
\ref{exyun}).

The paper is organized as follows: in Section \ref{preliminaries}
the notion of a quasi-regular DAE polynomial system is introduced,
along with some basic notions from differential algebra we use. In
Section \ref{defindex} we give the precise definition of the
$\fP$-differentiation index and show the relationship of this
invariant with the manifold of constrains of the system. The
Hilbert-Kolchin regularity and the order of the ideal $\fP$ are
analyzed in Section \ref{seccion hilbert kolchin}. Section
\ref{seccion jacobi} is devoted to the proofs of Jacobi-type bounds
for the differentiation index and the order. We present our results
on the differential membership problem for quasi-regular systems in
Section \ref{membership}. The last section of the paper is concerned
with existence and uniqueness of solution results for implicit
quasi-regular DAE systems.

\section{Preliminaries}\label{preliminaries}

\subsection{Basic Notations}
Let $K$ be a characteristic zero field equipped with a derivation
$\delta$. For instance $K=\mathbb{Q}$, $\mathbb{R}$ or $\mathbb{C}$
with $\delta:=0$, or $K=\mathbb{Q}(t)$ with the usual derivation
$\delta(t)=1$, etc.

For an arbitrary set of (differential) indeterminates $Z_1,\ldots
,Z_\alpha$ over $K$ we denote the $p$-th successive derivative of a
variable $Z_j$ as $Z_j^{(p)}$ (as customarily, the first derivatives
are also denoted by $\dot{Z_j}$); we write
$Z^{(p)}:=\{Z^{(p)}_1,\ldots,Z^{(p)}_\alpha\}$ and
$Z^{[p]}:=\{Z^{(i)},\ 0\le i\le p\}$.

The (non-noetherian) polynomial ring $K[Z^{(p)},\ p\in
\mathbb{N}_0]$, called the ring of \emph{differential polynomials},
is denoted by $K\{Z_1,\ldots ,Z_\alpha\}$ (or simply $K\{Z\}$).

Given a finite set of (differential) polynomials $H_1,\dots, H_\beta
\in K\{ Z_1,\dots, Z_\alpha\}$, we write $[H_1,\dots, H_\beta]$ to
denote the smallest ideal of $K\{Z_1,\dots, Z_\alpha\}$ stable under
differentiation, i.e. the smallest ideal containing $H_1,\dots,
H_\beta$ and all their derivatives of arbitrary order. The ideal
$[H_1,\dots, H_\beta]$ is called the \emph{differential ideal}
generated by $H_1,\dots, H_\beta$.

For  any differential polynomial $g$ lying in a differential
polynomial ring $K\{ Z_1,\ldots ,Z_\alpha \}$ the following
recursive relations hold for the successive total derivatives of
$g$:
$$\begin{array}{rcl}
g^{(0)}&:=& g, \\
g^{(p)}&:=& \delta(g^{(p-1)})+\displaystyle{\sum_{i\in
\mathbb{N}_0,1\le j\le \alpha} \dfrac{\partial g^{(p-1)}}{\partial
Z_j^{(i)}}Z_j^{(i+1)}}, \quad \hbox{ for } p\ge 1, \qquad
\end{array}$$ where $\delta(g^{(p-1)})$ denotes the polynomial
obtained from $g^{(p-1)}$ by applying the derivative $\delta$ to all
its coefficients (for instance, if $K$ is a field of constants, this
term is always zero).

\subsection{The system} \label{sistema original}

Let $r\le n\in \mathbb{N}$. Throughout the paper we consider DAE
systems of the following type:
\begin{equation} \label{sistema ampliado}
(\Sigma):=\left\{
\begin{array}
[c]{ccl}%
f_1(X_1^{[\epsilon_{11}]},\ldots, X_n^{[\epsilon_{1n}]}) &=& 0 \\
&\vdots &\\
f_r(X_1^{[\epsilon_{r1}]},\ldots, X_n^{[\epsilon_{rn}]}) &=& 0 \\
\end{array}
\right.
\end{equation}
where $f_1,\dots, f_r$ are differential polynomials in the $n$
differential variables $X:=X_1,\ldots, X_n$ with coefficients in the
field $K$. Each non-negative integer $\epsilon_ {ij}$ denotes the
maximal derivation order of the variable $X_j$ appearing in the
polynomial $f_i$. We denote $e:=\max\{\epsilon_{ij}\}$ for the
maximal derivation order which occurs in $(\Sigma)$. We assume that
$(\Sigma)$ actually involves derivatives, i.e. $e\ge 1$.

We write $[F]\subset K\{ X\}$ for the \emph{differential} ideal
generated by the polynomials $f_1, \dots, f_r$.

We introduce also the following auxiliary (noetherian) polynomial
rings and ideals: for every $i\in \NN_0$, $A_i$ denotes the
polynomial ring $A_i:=K[X^{[i]}]$ and $\Delta_i\subset A_{i-1+e}$
the ideal generated by the total derivatives of the defining
equations up to order $i-1$, namely $\Delta_i:=(f_1^{[i-1]},\ldots
,f_r^{[i-1]})$ (this ideal is usually known as the \emph{$i-1$
prolongation ideal}). We set $\Delta_0:=(0)$ by definition.

\subsection{Quasi-regular DAE systems}

The notion of \emph{quasi-regularity} appears implicitly
in \cite{johnson} in order to generalize
a Janet Conjecture to non-linear systems. Up to our knowledge it is the more
general frame where the Jacobi order bound holds
(see for instance \cite{Cohn2}, \cite{kondra}, \cite{ollisad}, \cite{ollivier}).

\begin{definition} \label{quasir}
If $\fP  \subset K\{X\}$ is a prime differential ideal
containing $f_1,\ldots,f_r$ (or equivalently, containing $[F]$) we
say that the system $(\Sigma)$ is \emph{quasi-regular at $\fP $} if
for every positive integer $i$, the Jacobian matrix of the
polynomials $f_1^{[i-1]},\ldots ,f_r^{[i-1]}$ with respect to the
set of variables $X^{[i-1+e]}$ has full row rank over the domain
$A_{i-1+e}/(A_{i-1+e}\cap\, \fP )$.
\end{definition}

This condition can be easily rephrased in terms of K\"ahler
differentials (as in Johnson's original work \cite{johnson})
saying that the differentials $\{
\textrm{d}f_j^{(i)},\ 1\le j\le r,\ i\in \mathbb{N}_0\} \subset
\Omega_{K\{X\}/K}$ are a $K\{X\}/\fP -$linearly independent set in
$\Omega_{K\{X\}/K}\otimes_{K\{X\}} K\{X\}/\fP $.\\

Geometrically, Definition \ref{quasir} means that for any positive integer $i$
the algebraic variety $V_i$ given by the ideal generated by the $ri$ polynomials
$f_1^{[i-1]},\ldots ,f_r^{[i-1]}$ in the $(i-1+e)n$-variate polynomial ring $A_{i-1+e}$
\emph{is smooth in almost all point} of the closed subvariety defined by
the prime ideal $A_{i-1+e}\cap\, \fP$.

This notion can also be interpreted in the context of jet spaces:
for the sake of simplicity, assume that $\fP$ is the prime ideal associated with a
solution of the ideal $[F]$ having infinitely many coordinates.
The quasi-regularity condition means that, if
$\varphi(t)$ is a solution of the differential system such that
$\varphi^{(j)}(t_0)\in \mathbb{R}^n$
are the $j$-th coordinates of the point defined by $\fP$, for each integer $i$
the jet of $\varphi$ at $t_0$ is a \emph{regular point} of the subspace
$V\subset \textbf{J}^\infty(\mathbb{R},\mathbb{R}^n)$
defined by the input system differentiated up to order $i$.

In the sequel, we will assume that not only a system $(\Sigma)$ is
given, but also a prime differential ideal $\fP \subset K\{X\}$
\emph{minimal} with the property of containing the polynomials
$f_1,\ldots,f_r$ and such that $(\Sigma)$ is quasi-regular at $\fP
$.\\

For each non-negative integer $i$ we write $B_i$ for the local ring
obtained from $A_i$ after localization at the prime ideal $A_i\cap
\fP $. Since each $A_i$ is a polynomial ring, the localizations $B_i$
are regular rings (see for instance \cite[\S 1, Corollary 1.8]{kunz2}).
We denote $\mathfrak{N}_i$ the maximal ideal generated by
$A_i\cap \fP $ in $B_i$. For the sake of simplicity we preserve the
notation $\Delta_i$ for the ideal generated by the derivatives up to
order $i-1$ of the polynomials $f_1,\ldots,f_r$ in the ring
$B_{i-1+e}$.

\begin{remark} \label{primo}
If the ideal $[F]\subset K\{X\}$ is already a prime ideal, the
minimality of $\fP$ implies that $\fP=[F]$ and all our results
remain true considering the rings $A_i$ and the ideal $[F]$ without
localization. In this case if $(\Sigma)$ is quasi-regular at $[F]$
we will say simply that $(\Sigma)$ is quasi-regular.
\end{remark}

We mention now some easy consequences of the quasi-regularity
hypothesis related to localization at the prime $\fP $:

\begin{proposition} \label{regular sequence en P}
Let $(\Sigma)$ be a DAE system defined by polynomials
$F:=f_1,\ldots,f_r$ whose orders are bounded by $e$ and let $\fP
\subset K\{X\}$ be a minimal prime differential ideal containing $F$
such that $(\Sigma)$ is quasi-regular at $\fP $. Let $i\in
\mathbb{N}$ be an arbitrary positive integer. Then we have:
\begin{enumerate}
\item $f_1^{[i-1]},\ldots,f_r^{[i-1]}$ is a regular sequence in the
local ring $B_{i-1+e}$ and generates a prime ideal.

\item In the
localized ring $K\{X\}_\fP $, the polynomials $f_1, \ldots , f_r$
generate the maximal ideal $\fP  K\{X\}_\fP $.

\item If $\mathbb{K}$ denotes the residual field of the prime ideal
$\fP $, the differential transcendence degree of $\mathbb{K}$ over
$K$ is $n-r$.
\end{enumerate}
\end{proposition}

\begin{proof}
The first assertion is a direct consequence of the classical
Jacobian Criterion (see for instance \cite[\S 29]{matsumura})
applied to the regular local ring $B_{i-1+e}$ and the polynomials
$f_1^{[i-1]},\ldots,f_r^{[i-1]}$ which belong to the maximal ideal
$\mathfrak{N}_{i-1+e}$. So, the quasi-regularity implies that these
polynomials form part of a system of parameters in $B_{i-1+e}$ and
in particular they generate a smooth (hence prime) ideal and form a
regular sequence.

Since only a finite number of variables are involved in primality
checking in $K\{X\}_\fP $, the previous arguments imply that the
differential ideal $[F]\, K\{X\}_\fP$ is necessarily prime in the
local ring $K\{X\}_\fP $ and then, by minimality, it agrees with
$\fP \, K\{X\}_\fP $.

For the last assertion we observe first that the polynomials
$f_1,\ldots, f_r$ are in fact differentially algebraically
independent over $K$: any finite family of them and their
derivatives is part of a system of parameters of a suitable local
$K$-algebra of type $B_i$ for all $i$ big enough, then
algebraically independent over $K$. Now, let $\theta_1,\ldots
,\theta_{n-r}\in K\{X\}$ be such that $f_1,\ldots,
f_r,\theta_1,\ldots ,\theta_{n-r}\in K\{X\}$ is a differential
transcendence basis of the fraction field of $K\{X\}$ over $K$.
Taking the classes of $\theta_1,\ldots ,\theta_{n-r}$ in the
fraction field of $(K\{X\}/[F])_\fP $ one obtains a differentially
algebraically independent family. Hence, the third assertion
follows from the second condition of the Proposition.
\end{proof}

\bigskip

{}From now on we will write $\Delta$ for the ideal $[F]\,
K\{X\}_\fP=\fP\, K\{X\}_\fP$.

\subsection{Pseudo-Jacobian matrices} \label{seccionjacobiana}

We introduce a family of pseudo-Jacobian matrices we need in order
to introduce the notion of differentiation index
in the next section:

\begin{definition} \label{defimatricespi}
For each $k\in\N$ and $i\in \mathbb{N}_{\ge e-1}$ (i.e. $i\in
\mathbb{Z}$ and $i\ge e-1$), we define the $kr\times kn$-matrix
$\mathfrak{J}_{k,i}$ as follows:
\[ \mathfrak{J}_{k,i}:= \left(
\begin{array}{cccc}
\frac{\partial F^{(i-e+1)}}{\partial {X}^{(i+1)}} & 0  & \cdots & 0
\\[1.5mm]
\frac{\partial {F}^{(i-e+2)}}{\partial {X}^{(i+1)}} & \frac{\partial
{F}^{(i-e+2)}}{\partial {X}^{(i+2)}} &
\cdots  & 0 \\[1.5mm]
\vdots & \vdots   & \ddots  &  \vdots \\[1.5mm]
\frac{\partial F^{(i-e+k)}}{\partial {X}^{(i+1)}}  & \frac{
\partial F^{(i-e+k)}}{\partial X^{(i+2)}}  &
\cdots & \frac{\partial F^{(i-e+k)}}{\partial X^{(i+k)}}
\end{array}
\right),
\]
where each $\frac{\partial F^{(p)}}{\partial {X}^{(q)}}$ denotes the
$r\times n-$block consisting in the Jacobian matrix of the
polynomials $f_1^{(p)},\ldots,f_r^{(p)}$ with respect to the
variables $X_1^{(q)},\ldots ,X_n^{(q)}$.
\end{definition}

Observe that the block triangular form of $\mathfrak{J}_{k,i}$
follows from the fact that the differential polynomials $F^{(i-e+p)}$
have order bounded by $i+p$. Hence, their derivatives with respect to
the variables $X^{(i+j)}$  are identically zero for $j\ge p+1$.

Definition \ref{defimatricespi} is quite ambiguous because this matrix may
be considered over different rings, and so, invariants as the rank
or the solution space are not well defined and in fact they may differ.

We add the following last hypothesis on our input system $(\Sigma)$:\\

\noindent \textbf{Hypothesis\ } We assume that for any pair of indices $i,k$ the rank of the
matrix $\mathfrak{J}_{k,i}$ over the integral domain $B_{i+k+s}/\Delta_{i-1+e+k+s}$ does not
depend on $s$. In other words, the rank of each matrix
$\mathfrak{J}_{k,i}$ considered alternatively over the rings $B_{i+k}/\Delta_{i-1+e+k}$,
or $A_{i+k}/(A_{i+k}\cap \fP )$, or the residual field of $B_{i+k}$, or the field
$\mathbb{K}:=K\{X\}/\fP$, is always the same.\\

This hypothesis is satisfied for relevant classes of DAE's: for instance, linear or linear
time varying DAE's, or DAE systems with generic second member coming from Control Theory
(see \cite{DJS07}). This kind of assumptions can be regarded as an algebraic counterpart
of usual hypotheses concerning solvability, smoothness and constant rank properties
for the system and its prolongations, required
in various notions of differentiation indices in the analytic-numerical framework
(see for instance \cite{campbell-gear} or \cite{brenan}).

\section{A notion of $\fP -$differentiation index}\label{defindex}

\subsection{A linear algebra--based definition}

In this section we introduce a notion of a \emph{differentiation
index} of the system $(\Sigma)$ depending on the chosen minimal
prime differential ideal $\fP $ containing it, always assuming that
the system is quasi-regular at this prime ideal and verifies the hypothesis
introduced in Subsection \ref{seccionjacobiana}.\\

We introduce a double sequence  $\mu_{k,i}$ of non-negative integers
associated with the matrices $\fJ_{k,i}$:

\begin{definition}
\label{emes} For $k\in \N_0$ and $i\in \mathbb{N}_{\ge e-1}$, we
define $\mu_{k, i}\in \N_0$ as follows:
\begin{itemize}
\item[--] $\mu_{0,i} := 0 $.
 \item[--]
$\mu_{k,i}:=\dim _{\mathbb{K}}\ker (\mathfrak{J}_{k,i}^t)$, for
$k\ge 1$, where $\fJ_{k,i}^t$ denotes the usual transpose of the
matrix $\fJ_{k,i}$. In particular $\mu_{k,i}=kr-{\rm
rank}_{\mathbb{K}} (\mathfrak{J}_{k,i})$.
\end{itemize}
\end{definition}

The sequence $\mu_{k,i}$ is strongly related with some algebraic
facts concerning the algebraic dimension of the ideals $\Delta_p$
generated by the first $(p-1)-$th total derivatives of the
polynomials $f_1,\ldots,f_r$:

\begin{proposition} \label{equival}
Let $k\in \N_0$ and $i\in \mathbb{N}_{\ge e-1}$. Then:
\begin{enumerate}
\item[(i)] The transcendence degree of the field extension
\[{\rm Frac}(B_i/(\Delta_{i-e+1+k}\cap B_i)) \hookrightarrow {\rm
Frac}(B_{i+k}/\Delta_{i-e+1+k})\] is $k(n-r)+\mu_{k,i}$.
\item[(ii)] The following identity holds:
$$ {\rm trdeg}_K {\rm Frac}(B_i/(\Delta_{i-e+1+k}\cap B_i))  =
(n-r)(i+1) + er - \mu_{k,i}.$$
\end{enumerate}
\end{proposition}

\begin{proof}The proof follows \cite[Proposition 2 and Remark 3]{DJS07}.

Consider the following diagram of field extensions:
\[\qquad
 \xymatrix{
{\rm Frac}(B_i/(\Delta_{i-e+1+k}\cap B_i)) \ar[r] & {\rm
Frac}(B_{i+k}/\Delta_{i-e+1+k})\\
K \ar[u]\ar[ru]& } \]

We can consider the fraction field of the
ring $B_{i+k}/\Delta_{i-e+1+k}$ as the fraction field of
$\Omega[X^{(i+1)},\ldots ,X^{(i+k)}]/(F^{(i)},\ldots ,F^{(i-e+k)})$,
where $\Omega$ is the fraction field of the integral domain  $B_i/(\Delta_{i-e+1+k}\cap B_i)$.
By the classical Jacobian Criterion (see for instance
\cite[Chapter 10, \S 27]{matsumura}), this transcendence degree equals
$kn - \textrm{rank}\ \mathfrak{J}_{k,i}=kn-(kr-\mu_{k,i})=k(n-r)+\mu_{k,i}$.

In order to prove the second assertion it suffices to compute
the transcendence degree of the fraction field of $B_{i+k}/\Delta_{i-e+1+k}$
over the ground field $K$.

Set $\mathfrak{p}:=A_{i+k}\cap \fP$ and let $\mathfrak{q}\subset \mathfrak{p}\subset A_{i+k}$ be the prime ideal
such that $\Delta_{i-e+1+k}B_{i+k}=\mathfrak{q}B_{i+k}$
(from the first assertion of Proposition \ref{regular sequence en P}
this prime ideal $\mathfrak{q}$ exists and it is unique
verifying these requirements).
Therefore, the transcendence degree of the
field extension $K\hookrightarrow {\rm
Frac}(B_{i+k}/\Delta_{i-e+1+k})$ is equal to the Krull dimension of
the domain $A_{i+k}/\mathfrak{q}$ and so, by the catenarity of the polynomial ring $A_{i+k}$,
equals $\dim A_{i+k}-\dim (A_{i+k})_{\mathfrak{q}}$. Again by the catenarity and the fact that
$\mathfrak{q}\subset \mathfrak{p}$ we conclude
$\dim (A_{i+k})_{\mathfrak{q}}=\dim (A_{i+k})_\mathfrak{p}- \dim (A_{i+k}/\mathfrak{q})_\mathfrak{p}$.
Since $ f_1^{[i-e+k]}, \ldots , f_r^{[i-e+k]}$ form a regular
sequence in $B_{i+k}$ (Proposition \ref{regular sequence en P})
we infer that $\dim (A_{i+k})_\mathfrak{p}- \dim (A_{i+k}/\mathfrak{q})_\mathfrak{p}$
agrees with the number
of such polynomials, i.e. with $r(i-e+1+k)$. Hence we obtain:
\[
\text{trdeg}_K {\rm
Frac}(B_{i+k}/\Delta_{i-e+1+k})= \dim A_{i+k} - r(i-e+1+k)=(i+k+1)(n-r)+er.
\]

The results follows directly from the additiveness of the transcendence degree. \end{proof}
\bigskip

By means of linear algebra arguments one can show as in
\cite[Section 3]{DJS07} some other relevant properties of the
sequence $\mu_{k,i}$; in particular we have the following results
which allow us to introduce the notion of \emph{differentiation
index} of $(\Sigma)$ related the quasi-regular prime differential
ideal $\fP$:

\begin{proposition} \label{mu no depende de i} \emph{{(\cite[Proposition 11]{DJS07})}}
The sequence $\mu_{k,i}$ does not depend on the index $i$.$\square$
\end{proposition}
In the sequel we write simply $\mu_k$ instead of $\mu_{k,i}$.

\begin{theorem}\label{definicion de indice}
\emph{(\cite[Theorem 9 and Corollary 14]{DJS07})} There exists
$\sigma\in \mathbb{N}_0$ such that $\mu_k <\mu_{k+1}$ for all
$k<\sigma$ and $\mu_k=\mu_{k+1}$ for all $k\ge \sigma$.$\square$
\end{theorem}

\begin{definition}
The constant $\sigma\in \mathbb{N}_0$ appearing in Theorem
\ref{definicion de indice} will be called \emph{the $\fP
$-differentiation index of the DAE system $(\Sigma)$}. If the
differential ideal generated by the equations of $(\Sigma)$ is
itself a prime ideal, we say simply \emph{the differentiation index}
of $(\Sigma)$.
\end{definition}

While the dependence of $\sigma$ on the prime ideal $\fP$ is obvious
from the construction, the name ``differentiation index" would seem
rather obscure at this point. This assignation will be justified
throughout the paper by proving that $\sigma$ verifies some
properties customarily associated to the differentiation index such
as estimating the number of derivatives needed either to obtain the
set of constraints that every solution of the system must satisfy
(Subsection \ref{secconstraints}) or to convert the DAE system
$(\Sigma)$ into an explicit ODE (Subsection \ref{ODEdim0}).

\subsection{Manifold of constraints}\label{secconstraints}

A remarkable property associated with most differentiation indices
is that they provide an upper bound for the number of derivatives of
the system needed to obtain all the constraints that must be
satisfied by  the solutions of the system (see for instance
\cite{Fliess}).

For instance, suppose that the input system is of first order and we look for
 the purely polynomial conditions in
$K[X_1,\ldots ,X_n]$ (i.e. no derivatives of the variables appear)
which must be verified \emph{by any solution}. These conditions
define an algebraic variety which is usually called the
\emph{manifold of constraints} of the system. Clearly this variety
is defined by the ideal $[F]\cap K[X_1,\ldots ,X_n]$. From the
noetherianity of the polynomial ring, there exists the smallest non
negative integer $s$ such that the elements of this ideal may be
written using at most $s$ derivatives of the input equations. Under
our conditions of Section \ref{preliminaries}, this minimum integer
is exactly the differentiation index:

\begin{theorem} \label{sadik}
Let $\sigma \in \N_{0}$ be the $\fP$-differentiation index of the
system (\ref{sistema ampliado}). Then, for every $i\in \N_{\ge
e-1}$, the equality of ideals $$\Delta_{i-e+1+\sigma}\cap B_i=\Delta
\cap B_i$$  holds in the ring $B_i$.
 Furthermore, for every $i\in \N_{\ge e-1}$, the
$\fP$-differentiation index $\sigma$ verifies: $\sigma= \min\{ h\in
\N_0 : \Delta_{i-e+1+h} \cap B_i = \Delta \cap B_i\}.$
\end{theorem}

\begin{proof} Fix an index $i\in \N_{\ge e-1}$.
Let us consider the increasing chain $(\Delta_{i-e+1+k}\cap
B_i)_{k\in \N_0}$ of prime ideals in the  ring $B_i$. {}From
Proposition \ref{equival}, for every $k\in \N_0$, we have
\begin{equation}\label{trdegmu}
{\rm trdeg}_K ({\rm Frac}(B_i/\Delta_{i-e+1+k}\cap B_i))
 =  (n-r)(i+1) + er - \mu_{k}.
\end{equation}
Since $\mu_{k}$ is stationary for $k\ge \sigma$ (Theorem
\ref{definicion de indice}), all the prime ideals
$\Delta_{i-e+1+k}\cap B_i$ have the same dimension for $k\ge \sigma$
and the chain of prime ideals becomes stationary for $k\ge \sigma$.

It only remains to prove that the largest ideal of the chain
coincides with $\Delta \cap B_i$. One inclusion is obvious. For the
other, let $f$ be an arbitrary element of $\Delta \cap B_i$, then
there exist differential polynomials $h, a_{ij}\in k\{X\}$, $h\notin
\fP$  such that
\[f(X^{[i]})=\sum_{i=1}^r\sum_{j}\frac{a_{ij}f_i^{(j)}}{h}.\]
If  $N$ is an integer greater than $i+\sigma$ and than  the maximal
order of the variables $X$ appearing in this equality, we have $f\in
B_i$ and $f\in \Delta_{N-e+1} \subset  B_N$. Then $f\in
\Delta_{N-e+1}\cap B_i$ and, since the above chain of ideals is
stationary, $f\in \Delta_{i-e+1+\sigma}\cap B_i$. This finishes the
proof of the first assertion of the Theorem.

In order to prove the second part of the statement, for each $i\in
\mathbb{N}_{\ge e-1}$, let $h_i$ be the smallest non-negative
integer such that $\Delta_{i-e+1+h_i}\cap B_i  = \Delta\cap B_i $.
By the definition of $h_i$, the transcendence degrees ${\rm trdeg}_K
({\rm Frac}(B_i/\Delta_{i-e+1+k}\cap B_i))$ coincide for $k\ge h_i$,
and so, $\mu_{k} $ is constant for $k\ge h_i$ (see identity
(\ref{trdegmu}) above). This implies that $\sigma  \le h_i$. The
equality follows from the first part of the statement and the
minimality of $h_i$.
\end{proof}

\begin{remark}
Taking $i=e-1$ in the last assertion of  Theorem \ref{sadik}, we
have the following alternative definition of the
$\fP$-differentiation index: $$ \sigma= \min\{ h\in \N_0 :
\Delta_{h} \cap B_{e-1} = \Delta \cap B_{e-1}\}.$$
\end{remark}

\section{The Hilbert-Kolchin polynomial of the ideal
$\fP$} \label{seccion hilbert kolchin}

Let $\fP \subset K\{ X \}$ be a minimal prime differential ideal
such that the system $(\Sigma)$ is quasi-regular at $\fP$. The last
assertion in Proposition \ref{regular sequence en P} states that the
differential dimension of the prime differential ideal $\fP$ is
$n-r$. Following \cite[Chapter II, Section 12, Theorem 6]{kol}, the
transcendence degree of the fraction field of the domain
$A_i/(A_i\cap \fP)$ over the ground field $K$ equals $(n-r)(i+1)+c$
for all $i$ sufficiently big, where $c$ is a non-negative integer
constant which depends only on the ideal $\fP$. This constant $c$ is
called \emph{the order of $\fP$} and denoted by $\text{ord}(\fP)$.

The polynomial $H_\fP(T):=(n-r)(T+1)+\text{ord}(\fP)$ is called the
Hilbert-Kolchin polynomial of $\fP$. The minimum of the indices
$i_0$ where $H_\fP(i)=\textrm{trdeg}_K\textrm{Frac}(A_i/(A_i\cap
\fP))$ for all $i\ge i_0$ is known as \emph{the Hilbert-Kolchin
regularity of the ideal $\fP$}.

\subsection{The Hilbert-Kolchin regularity}

A well-known result from the theory of \emph{characteristic sets}
states that the Hilbert-Kolchin regularity of a prime differential
ideal is equal to $\max \{\text{ord}(C) : C \in \mathcal{C}\}-1$
where $\mathcal{C}$ is a characteristic set of the differential
ideal  for an \emph{orderly ranking} (see \cite[Ch.~II, Section 12,
Th.~6 (d)]{kol} and \cite[Theorem 3.3]{carra}).

The results developed so far enable us to exhibit the following
simple upper bound for the Hilbert-Kolchin regularity of the ideal
$\fP$ depending only on the maximal order of derivation of the
variables involved in the system (\ref{sistema ampliado}),
independently of knowing a characteristic set:

\begin{theorem} \label{hilbert constante}
The Hilbert-Kolchin regularity of the ideal $\fP$ is bounded by
$e-1$.

\noindent In particular, for first-order systems of type
(\ref{sistema ampliado}) (in other words, for the case $e=1$), ${\rm
trdeg}_K{\rm Frac}(A_i/(A_i\cap \fP))=(n-r)(i+1)+{\rm ord}(\fP)$ for
all $i\in \NN_0$.
\end{theorem}

\begin{proof}
Since for all $i\in \mathbb{N}_0$ the fraction field of the domain
$A_i/(A_i\cap \fP)$ coincides with the residual field of the local
ring obtained from $A_i$ after localization at the prime ideal
$A_i\cap \fP$, we have that $\textrm{Frac}(A_i/(A_i\cap
\fP))=\textrm{Frac}(B_i/(B_i\cap \Delta))$. 
Then, we have that $\textrm{trdeg}_K\textrm{Frac}(A_i/(A_i\cap
\fP))= \textrm{trdeg}_K\textrm{Frac}(B_i/\Delta \cap B_i)$ and so,
it is enough to show that, for all $i\ge e-1$,
$\textrm{trdeg}_K\textrm{Frac}(B_i/\Delta \cap
B_i)+n-r=\textrm{trdeg}_K\textrm{Frac}(B_{i+1}/\Delta \cap
B_{i+1})$.

Fix an index $i\ge e-1$.  Due to Theorem \ref{sadik}, we have that
$\Delta \cap B_i = \Delta_{i+1-e+\sigma} \cap B_i$ and $\Delta \cap
B_{i+1} = \Delta_{i+2-e+\sigma} \cap B_{i+1}$. Thus, using
Proposition \ref{equival}  we obtain:
\begin{eqnarray*}
\textrm{trdeg}_K\textrm{Frac}(B_{i+1}/(B_{i+1}\cap \Delta)) &=& (n-r)(i+2) + er - \mu_{\sigma},\\
\textrm{trdeg}_K\textrm{Frac}(B_i/(B_i\cap \Delta)) &=& (n-r)(i+1) +
er -\mu_{\sigma}.
\end{eqnarray*}
Hence, the result holds.
\end{proof}

\bigskip

We point out that as a consequence of this theorem we deduce that
the order of any characteristic set of $\fP$ for an orderly ranking
is bounded by the maximal order $e$.

\subsection{The order}

As a consequence of the results of the previous sections we are able
to prove the following characterization for the order of the ideal
$\fP$.

\begin{proposition} \label{ordenmu} Let $\fP\subset K\{X_1\dots, X_n\}$ be a minimal prime differential ideal containing
polynomials  $f_1,\dots, f_r\in K\{X_1\dots, X_n\}$. Assume that the
DAE system defined by $f_1,\dots, f_r$ is quasi-regular at $\fP$.
Then, ${\rm ord}(\fP)= er-\mu_\sigma$.
\end{proposition}

\begin{proof}
Let $i_0$ be the regularity of the Hilbert-Kolchin function of
$\fP$. Then, for every $i\ge i_0$, the polynomial $H_\fP(i)$ agrees
with the Hilbert-Kolchin function of $\fP$, that is $
(n-r)(i+1)+\textrm{ord}(\fP) =
\textrm{trdeg}_K\textrm{Frac}(A_i/(\fP \cap A_i))$. Localizing at
the prime ideal $\fP \cap A_i$, we have that $\textrm{Frac}(A_i/(\fP
\cap A_i)) = B_i/\mathfrak{N}_i = B_i /(\Delta \cap B_i)$ for every
$i\ge i_0$, where $\Delta$ is the differential ideal generated by
$f_1,\dots, f_r$ in $K\{X\}$.

Since $B_{i_0}$ is a Noetherian ring, for $h\in \N_0$ sufficiently
big, we have that the equality $\Delta \cap B_{i_0} = \Delta_h \cap
B_{i_0}$ holds and so, $ (n-r)(i_0+1)+\textrm{ord}(\fP) =
\textrm{trdeg}_K(B_{i_0}/(\Delta_h\cap B_{i_0}))$. Now, by
Proposition \ref{equival}, $\textrm{trdeg}_K(B_{i_0}/(\Delta_h\cap
B_{i_0}) )= (n-r) (i_0+1) +er -\mu_{h-i_0+e-1}$, which implies that
$\textrm{ord}(\fP) =er -\mu_{h-i_0+e-1}$ for $h$ sufficiently big.
Therefore from Theorem \ref{definicion de indice} we have that
$\textrm{ord}(\fP) =er -\mu_{\sigma}$.
\end{proof}

\section{Jacobi-type bounds} \label{seccion jacobi}

In \cite{jacobi} and \cite{jacobi2}, Jacobi introduces a parameter
associated to the orders of derivations in a DAE system and
conjectures an upper bound for the order of the system in terms of
this number.

\begin{definition}\label{Jacobinumber} Let $A \in \N_0^{s\times m}$, $s\le m$, be an
integer matrix. The \emph{Jacobi number} of $A$ is defined to be
$$J(A):= \max\Big\{\sum_{i=1}^s a_{i\tau(i)} \mid \tau:\{1,\dots, s\} \to \{1,\dots,
m\} \text{ is an injection}\Big\}.$$
\end{definition}

Jacobi asserts that the order of a DAE system $(\Sigma)$ with $r=n$
equations is bounded above by $J(\mathcal{E})$, where
$\mathcal{E}:=(\epsilon_{ij})_{1\le i, j\le n}$ is the matrix whose
entries are the maximal derivation orders $\epsilon_{ij}$ of the
variables $X_j$ appearing in the polynomials $f_i$ and
$-\infty$ whenever the variable $X_j$ does not appear in $f_i$.

We consider also the integer matrix ${\mathcal E}_0$ which consists in the previous matrix
$\mathcal{E}$ setting $0$ instead of each $-\infty$.
Clearly $J({\mathcal E})\le J({\mathcal E}_0)$.

In this section, we will show that the sum of differentiation index of
a quasi-regular DAE system at a prime differential ideal $\fP$ and
the order of  $\fP$ can be bounded in terms of the maximum and
minimum orders of the equations and the Jacobi number
$J(\mathcal{E}_0)$. With the same techniques we can also recover the Jacobi
bound for the order of the system $(\Sigma)$ (see Theorem \ref{cotaorden} below).

Our main result of this section is the following:

\begin{theorem}\label{indexbound}
Let $(\Sigma)$ be a DAE system defined by polynomials $f_1,\dots,
f_r \in K\{X_1,\dots, X_n\}$  and let $\fP$ be a minimal prime
differential ideal containing $f_1,\dots, f_r$ such that $(\Sigma)$
is quasi-regular at $\fP$. Consider the integer matrix
$\mathcal{E}_0:=(\epsilon_{ij})_{1\le i\le r, 1\le j \le n}$, where
$\epsilon_{ij} := {\rm ord}_{X_j}(f_i)$ and $0$ if the variable $X_j$ does not appear in $f_i$.
Then, the
$\fP$-differentiation index $\sigma$ of the system $(\Sigma)$ and
the order ${\rm ord}(\fP)$ of the differential ideal $\fP$ satisfy
$$\sigma + {\rm ord}(\fP) \le J(\mathcal{E}_0) + \max\{\epsilon_{ij}\}
- \min\{\epsilon_{ij}\}.$$
\end{theorem}

Before proving the Theorem \ref{indexbound}, we will show some
auxiliary technical results concerning the Jacobi number of integer
matrices.

\bigskip

Let $L$ be an arbitrary field. Given a matrix $A = (a_{ij}) \in L^{s\times m}$, we can estimate its
rank by means of the Jacobi number of an associated binary matrix:
we define $B(A) \in L^{s \times m}$ as the matrix whose entries are
$$
B(A)_{ij} = \begin{cases} 1 & \textrm{ if }  a_{ij} \ne 0 \cr 0 &
\textrm{ if }  a_{ij} = 0
\end{cases}
$$

\begin{lemma}\label{jacobi binary}
Let $A \in L^{s \times m}$ with $s\le m$. Then ${\rm rank}(A) \le
J(B(A))$.
\end{lemma}

\begin{proof} Set $J:= J(B(A))$. By definition, $J = \max_{\tau}
\{ \sum_{1\le h \le s} B(A)_{h \tau(h)}\}$, where the maximum is
taken over all injections $\tau:\{1,\dots, s\}\to \{ 1,\dots,
m\}$. Since $B(A)_{ij} =0$ or $1$ according as $ a_{ij}$ equals zero
or not, it follows that for every $\tau$, the tuple $(
a_{1\tau(1)}, \dots, a_{s \tau(s)})$ has at most $J$ nonzero
entries.

Let $k$ with $J<k\le s$ and choose indices $1\le i_1<\dots< i_k\le
s$ and $1\le j_1<\dots< j_k\le m$. Consider the $k\times k$ square
submatrix of $A$ corresponding to the intersection of the rows
indexed by $i_1,\dots, i_k$ and the columns indexed by $j_1,\dots,
j_k$. The determinant of this submatrix equals $\sum_{\nu} {\rm
sg}(\nu) \prod_{1\le h\le k} a_{i_h \nu(i_h)}$, where the sum runs
over all bijections $\nu : \{i_1,\dots, i_k\} \to \{j_1,\dots,
j_k\}$. Now, for each $\nu$, the vector $(a_{i_1 \nu(i_1)}, \dots,
a_{i_k \nu(i_k)})$ has at least one zero entry (since $k> J$ and
every $\nu$ can be extended to an injection $\tau$) and so, the
determinant vanishes. We conclude that ${\rm rank}(A) \le J$.
\end{proof}

\bigskip

Our second auxiliary result relates the Jacobi number of an integer
matrix to the Jacobi number of an associated binary matrix of a
particular kind.

Fix $k\in \N$. For every $0\le a\le k$, let $T_{k, a}\in
\N_0^{k\times k}$ be the lower triangular matrix defined by
$$(T_{k,a})_{ij}=\begin{cases} 1 & {\rm if } \ k-a \le i-j \le k-1\cr 0 &
{\rm otherwise},
\end{cases}
$$
i.e. $T_{k,a}$ has the $a$ lower diagonals with all their entries
equal to $1$ and all the remaining entries of the matrix are zero:

{\footnotesize
\begin{align*}
&
\genfrac{}{}{-1.5pt}{}{{}}{a\left\{
\begin{tabular}
[c]{p{-7cm}}%
\\
\multicolumn{1}{c}{}\\
\multicolumn{1}{c}{}\\
\multicolumn{1}{c}{}\\
\multicolumn{1}{c}{}\\
\end{tabular}
\right.  }%
\!\!\!\left(
\begin{tabular}
[c]{ccccccc}%
$0$ & $\cdots$ & $\cdots$ &  &  & $\cdots$ & $0$\\
$\vdots$ & $\ddots$ &  &  &  &  & $\vdots$\\
$0$ &  & $\ddots$ &  &  &  & \\
\cline{1-4}\cline{2-3}%
\multicolumn{1}{|c}{$1$} & $0$ & $\cdots$ & $0$ &
\multicolumn{1}{|c}{$\cdots
$} & $\cdots$ & \\
\multicolumn{1}{|c}{$1$} & $\ddots$ & $\ddots$ &  &
\multicolumn{1}{|c}{$\ddots$} &  & $\vdots$\\
\multicolumn{1}{|c}{$\vdots$} & $\ddots$ & $\ddots$ & $0$ &
\multicolumn{1}{|c}{} & $\ddots$ & $\vdots$\\
\multicolumn{1}{|c}{$1$} & $\cdots$ & $1$ & $1$ &
\multicolumn{1}{|c}{$0$} &
$\cdots$ & $0$\\
\cline{1-4}%
\end{tabular}
\ \ \right)\ .  \\
\qquad \qquad & \qquad \quad
\begin{tabular}
[c]{c}{}%
$\underbrace{\ \qquad \qquad \qquad \qquad \quad }$\\
$a$
\end{tabular}
\end{align*}
}

Given a matrix $A  = (a_{ij})\in \N_0^{s\times m}$ and a positive
integer $k\ge \max_{i, j} \{a_{ij}\}$, we define a $ks \times km$
block binary matrix $T_k(A)$ as follows:
$$T_k(A) = \begin{pmatrix} T_{k, a_{11}} & \dots & T_{k, a_{1m}} \cr
\vdots & & \vdots \cr T_{k, a_{s1}} & \dots &T_{k,
a_{sm}}\end{pmatrix}.$$

\begin{lemma}\label{jacobi triangular}
Let $A \in \N_0^{s\times m}$ with $s\le m$. Then, for every $k\ge
\max_{i,j} \{a_{ij}\}$, we have $J(A) = J(T_k(A))$.
\end{lemma}

\begin{proof}
First, note that it suffices to prove the identity for square
matrices: in any case, one can make $A$ square by adding null rows,
which does not change its Jacobi number or the Jacobi number of its
expansion $T_k(A)$. So, we assume $A \in \N_0^{m\times m}$.

By definition, $J(A) = \max_{\tau \in S_m} \{ \sum_{1\le i \le m}
a_{i \tau(i)} \}$. Without loss of generality, we may assume that
the maximum is attained when $\tau$ is the identity. We consider a
permutation $\widehat \tau\in S_{km}$ which chooses a $km$-tuple
of entries of $T_k(A)$ consisting of the $a_{ii}$ $1$'s of the upper
nonzero diagonal of each block $T_{k, a_{ii}}$ in the main diagonal
of $T_k(A)$ along with some zero elements: for instance, the
permutation defined by $\widehat \tau((i-1)k+h) = (i-1)
k+r_k(h+a_{ii}-1)+1$ for every $1\le i \le m$, $1\le h\le m$, where
$r_k$ denotes remainder in the division by $k$. Then,
$\sum_{l=1}^{km} T_k(A)_{l \widehat\tau(l)} = \sum_{i=1}^m a_{ii}
=J(A)$. It follows that $J(A) \le \max_{\widehat \tau \in
S_{km}}\sum_{l=1}^{km} T_k(A)_{l \widehat\tau(l)} =J(T_k(A))$.

In order to prove the other inequality we will use the following
result known as the K\"onig-Egerv\'ary theorem (see \cite{egervary},
\cite{konig}): \begin{equation}\label{egervary}J(A) = \min \Big\{
\sum_{i=1}^m \lambda_i + \sum_{j=1}^m \phi_j \mid \lambda_i, \phi_j
\in \N_0,\, \lambda_i +\phi_j \ge a_{ij} \ \forall\, 1\le i, j \le
m\Big \}.\end{equation} Consider nonnegative integer vectors
$\lambda =(\lambda_1, \dots, \lambda_m)$ and $\phi =(\phi_1, \dots,
\phi_m)$ such that $\lambda_i + \phi_j \ge a_{ij}$ for every $1\le
i, j\le m$ and $J(A) = \sum_{i=1}^m \lambda_i + \sum_{j=1}^m
\phi_j$.

We define nonnegative integer vectors $\Lambda= (\Lambda_1,\dots,
\Lambda_m), \Phi= (\Phi_1, \dots, \Phi_m)\in \N_0^{km}$ with
$\Lambda_i , \Phi_j \in \N_0^k$ as follows: for every $1\le i \le m$
and $1\le j \le m$,
$$ \Lambda_{ih} = \begin{cases} 0 & \hbox{ for } h=1,\dots, k-\lambda_i\\
1 & \hbox{ for } h=k-\lambda_i+1,\dots, k
\end{cases} \quad \hbox{ and } \quad
\Phi_{jl} = \begin{cases} 1 & \hbox{ for } l=1,\dots, \phi_j\\
0 & \hbox{ for } l=\phi_j+1,\dots, k.
\end{cases}
$$
Let us show that $\Lambda, \Phi$ satisfy $\Lambda_{ih} + \Phi_{jl}
\ge T_k(A)_{(i-1)k +h, (j-1)k+l}$, or equivalently,

\begin{equation} \label{tata}
\Lambda_{ih}+ \Phi_{jl}\ge (T_{k, a_{ij}})_{hl}
\end{equation}
for every $1\le i, j \le m$ and $ 1\le h, l\le k$.

Recall that $(T_{k, a_{ij}})_{hl}$ is nonzero only for those $h, l$
with $k-a_{ij} \le h-l \le k-1$ and, in these cases, it equals $1$.

If $h\ge k-\lambda_i+1$ or $l\le \phi_j$, we have $\Lambda_{ih}+
\Phi_{jl}\ge 1$, since at least one of the terms equals $1$ and so,
inequality (\ref{tata}) holds in this case. Now, if $h\le
k-\lambda_i$ and $l\ge \phi_j -1$, both $\Lambda_{ih}$ and
$\Phi_{jl}$ are zero, but $h-l \le k-(\lambda_i +\phi_j)-1 \le
k-a_{ij} -1$, which implies that $(T_{k, a_{ij}})_{hl} = 0$.

Finally, since
$$\sum_{i=1}^m \sum_{h=1}^k \Lambda_{ih} +
\sum_{j=1}^m \sum_{h=1}^k  \Phi_{jl} = \sum_{i=1}^m \lambda_i +
\sum_{j=1}^m \phi_j = J(A),$$ applying identity (\ref{egervary}) to
the matrix $T_k(A)$, we conclude that $J(T_k(A)) \le J(A)$.
\end{proof}

\bigskip

We are now ready to prove the main result of this section.

\bigskip

\begin{proof}
\textit{(Proof of Theorem \ref{indexbound}.)} We will show that
\begin{equation}\label{cota mus}
 er - J(\mathcal{E}_0) \le \mu_k\le er - {\rm ord}(\fP) \qquad \forall\, k\ge
e - \min\{\epsilon_{ij}\},
\end{equation}
where $e=\max\{\epsilon_{ij}\}$.

These inequalities imply that the sequence $(\mu_k)_{k\ge e-\min\{\epsilon_{ij}\}}$
increases at most $(er-{\rm ord}(\fP))-(er-J(\mathcal{E}_0))=
J(\mathcal{E}_0)-{\rm ord}(\fP)$ times and hence the complete sequence
$(\mu_k)_{k\in\mathbb{N}_0}$ increases at most
$e-\min\{\epsilon_{ij}\}+ J(\mathcal{E}_0)-{\rm ord}(\fP)$ times.
By Theorem \ref{definicion de indice} this number bounds
the differentiation index $\sigma$ and Theorem \ref{indexbound} follows.


Then, it suffices to prove the inequality (\ref{cota mus}).

First, note that by Theorem \ref{definicion de indice} and
Proposition \ref{ordenmu}, the inequality $\mu_k \le \mu_\sigma = er
- {\rm ord}(\fP)$ holds for every $k\in \N_0$, which proves the
second inequality.

Now, fix $k\ge e- \min\{\epsilon_{ij}\}$. By Defintion \ref{emes},
$\mu_k = kr - \textrm{rank}(\mathfrak{J}_{k,e-1})$. In
order to simplify notation, we will write
$\fJ_k:=\mathfrak{J}_{k,e-1}$. Let us consider the matrix
$\widetilde \fJ_k$ which is obtained by permutation of rows and
columns of $\fJ_k$ so that:
$$\widetilde\fJ_k = \begin{pmatrix}
\frac{\partial f_1^{[k-1]}}{\partial X_1^{[e, e+k-1]}} &
\frac{\partial f_1^{[k-1]}}{\partial X_2^{[e, e+k-1]}} & \cdots &
\frac{\partial f_1^{[k-1]}}{\partial X_n^{[e, e+k-1]}}\\
\frac{\partial f_2^{[k-1]}}{\partial X_1^{[e, e+k-1]}} &
\frac{\partial f_2^{[k-1]}}{\partial X_2^{[e, e+k-1]}} & \cdots &
\frac{\partial f_2^{[k-1]}}{\partial X_n^{[e, e+k-1]}}\\
\vdots & \vdots& & \vdots \\
\frac{\partial f_r^{[k-1]}}{\partial X_1^{[e, e+k-1]}} &
\frac{\partial f_r^{[k-1]}}{\partial X_2^{[e, e+k-1]}} & \cdots &
\frac{\partial f_r^{[k-1]}}{\partial X_n^{[e, e+k-1]}}
\end{pmatrix},$$
where for every $1\le i \le r$ and $1\le j \le n$,
\begin{equation} \label{block}
\frac{\partial f_i^{[k-1]}}{\partial X_j^{[e,\, e+k-1]}}=
\left(
\begin{array}{cccc}
\frac{\partial f_i}{\partial X_j^{(e)}} & 0  & \cdots & 0
\\[3mm]
\frac{\partial \dot f_i}{\partial X_j^{(e)}} & \frac{\partial \dot
f_i}{\partial X_j^{(e+1)}} &
\cdots  & 0 \\[1.5mm]
\vdots & \vdots   & \ddots  &  \vdots \\[1.5mm]
\frac{\partial f_i^{(k-1)}}{\partial X_j^{(e)}}  & \frac{
\partial f_i^{(k-1)}}{\partial X_j^{(e+1)}}  &
\cdots & \frac{\partial f_i^{(k-1)}}{\partial X_j^{(e+k-1)}}
\end{array}
\right)\in \mathbb{K}^{k\times k}.
\end{equation}

Applying Lemma \ref{jacobi binary}, we have that
$\textrm{rank}(\fJ_k) = \textrm{rank}(\widetilde \fJ_k) \le
J(B(\widetilde \fJ_k))$. We are now going to estimate the Jacobi
number of $B(\widetilde \fJ_k)$.

As ${\rm ord}_{X_j}(f_i) = \epsilon_{ij}$, the partial derivative
$\frac{\partial f_i^{(p)}}{\partial X_j^{(q)}}$ is identically zero
for all $p, q$ with $q-p > \epsilon_{ij}$. Thus, the only nonzero
entries of the block $\frac{\partial f_i^{[k-1]}}{\partial
X_j^{[e,\, e+k-1]}}$ or, equivalently, all the $1$'s of the
corresponding block in the binary matrix $B(\widetilde \fJ_k)$, lie
in places $(h,l)$ with $h-l\ge e-\epsilon_{ij}$.

Taking into account that the Jacobi number of a matrix does not
decrease by replacing some of its zero entries by positive integers,
we deduce that the Jacobi number of $B(\widetilde \fJ_k)$ is bounded
above by the Jacobi number of the block binary matrix which is
obtained by setting to $1$ each entry of $B(\widetilde \fJ_k)$
corresponding to all these partial derivatives that could be
nonzero:

\begin{equation} \label{triangblock}
\begin{pmatrix} T_{k, k-e+\epsilon_{11}} & T_{k, k-e+\epsilon_{12}} & \dots & T_{k,
k-e+\epsilon_{1n}} \\
T_{k, k-e+\epsilon_{21}} & T_{k, k-e+\epsilon_{22}} & \dots & T_{k,
k-e+\epsilon_{2n}} \\
\vdots & & \vdots\\
T_{k, k-e+\epsilon_{r1}} & T_{k, k-e+\epsilon_{r2}} &\dots & T_{k,
k-e+\epsilon_{rn}}
\end{pmatrix}.
\end{equation}

Note that the above matrix is the binary matrix $T(\mathcal{S}_k)$
associated with the integer matrix
$$\mathcal{S}_k := \begin{pmatrix}
k-e+\epsilon_{11} & k-e+\epsilon_{12}&\dots & k-e+\epsilon_{1n}\\
k-e+\epsilon_{21} & k-e+\epsilon_{22}&\dots & k-e+\epsilon_{2n}\\
\vdots & \vdots & & \vdots \\
k-e+\epsilon_{r1} &
k-e+\epsilon_{r2}&\dots & k-e+\epsilon_{rn}
\end{pmatrix}.$$
Therefore, due to Lemma \ref{jacobi triangular}, $J(B(\widetilde
\fJ_k))\le J(T(\mathcal{S}_k))= J(\mathcal{S}_k) = \max_\tau \{
\sum_{i=1}^r k-e+\epsilon_{i \tau(i)} \} = (k-e)r + \max_\tau \{
\sum_{i=1}^r \epsilon_{i \tau(i)} \} = (k-e) r + J(\mathcal{E}_0)$.

We conclude that $\textrm{rank}(\fJ_k) \le (k-e) r +J(\mathcal{E}_0)$
and so, $\mu_k = kr -\textrm{rank}(\fJ_k) \ge er - J(\mathcal{E}_0)$.
\end{proof}

\bigskip

\noindent \textbf{Remark on the proof\ }
Suppose that the variable $X_j$ does not occur in the equation $f_i$;
then the matrix (\ref{block}) is zero. Hence, the corresponding
triangular block $T_{k,k-e+\epsilon_{ij}}$ in (\ref{triangblock}) may be
taken as the zero matrix and we can replace the integer
$k-e+\epsilon_{ij}$ by $0$ in the matrix $\mathcal{S}_k$,
always preserving the inequality $\textrm{rank}(\fJ_k)
\le J(T(\mathcal{S}_k))= J(\mathcal{S}_k)$ for the new matrix
$\mathcal{S}_k$. Therefore, if we redefine $\epsilon_{ij}:=-\infty$
for those pairs $(i,j)$ such that $X_j$ does not appear in $f_i$, we
conclude that the inequality
\begin{equation} \label{infinite}
\textrm{rank}(\fJ_k)\le \max_\tau \{
\sum_{i=1}^r \max\{0,k-e+\epsilon_{i \tau(i)}\} \}.
\end{equation}
holds for any integer $k$ verifying $k\ge e-\min\{\epsilon_{ij}\}$. \\

This remark is a key fact we use to obtain the Jacobi bound for the order
of the prime ideal $\fP$:

\begin{theorem} \label{cotaorden} Let $\fP\subset K\{X_1\dots, X_n\}$ be a minimal prime
differential ideal containing polynomials  $f_1,\dots, f_r\in
K\{X_1\dots, X_n\}$. Assume that the DAE system defined by
$f_1,\dots, f_r$ is quasi-regular at $\fP$ and verifies the hypothesis
of Subsection \ref{seccionjacobiana}. Let $\mathcal{E}:=
(\epsilon_{ij})_{1\le i\le r,1\le j \le n}$, where
$\epsilon_{ij}:={\rm ord}_{X_j}(f_i)$
and $-\infty$ if the variable $X_j$ does not appear in $f_i$.
Then, ${\rm ord}(\fP) \le J(\mathcal{E})$.
\end{theorem}

\begin{proof}
It suffices to show that
the inequality (\ref{cota mus}) holds for some index $k$ after
replacing
the matrix $\mathcal{E}_0$ by $\mathcal{E}$. Since the inequality
$\mu_k\le er- {\rm ord}(\fP)$ is independent of the matrix $\mathcal{E}$ it
remains to prove that there exists an integer $k$ such that
$er -J(\mathcal{E})\le \mu_k$ holds. From the definition of $\mu_k$
this inequality is equivalent to
$\textrm{rank}(\fJ_k) \le (k-e) r +J(\mathcal{E})$.

From (\ref{infinite}) in the previous remark it is enough to prove that
\[
\max_\tau \{
\sum_{i=1}^r \max\{0,k-e+\epsilon_{i \tau(i)}\} \}\le (k-e) r +J(\mathcal{E})
\]
holds for some $k\ge e-\min\{\epsilon_{ij}\}$.

Take any $k\ge er$ and let $\tau :\{1,\ldots,r\}\rightarrow \{1,\ldots ,n\}$
injective.

First, suppose that $\epsilon_{i\tau(i)}\ne -\infty$ for all $i=1,\ldots ,r$. Then
$\max\{0,k-e+\epsilon_{i \tau(i)}\}=k-e+\epsilon_{i \tau(i)}$ for all $i$. Hence
$\sum_{i=1}^r \max\{0,k-e+\epsilon_{i \tau(i)}\}=(k-e)r+ \sum_{i=1}^r \epsilon_{i \tau(i)}$,
which is clearly bounded by $(k-e) r +J(\mathcal{E})$.

Now, suppose that $\epsilon_{i_0\tau(i_0)}= -\infty$ for some index $i_0$.
Then we have
\begin{equation} \label{tau}
\sum_{i=1}^r \max\{0,k-e+\epsilon_{i \tau(i)}\}=\sum_{i\ne i_0} \max\{0,k-e+\epsilon_{i \tau(i)}\}
\le k(r-1),
\end{equation}
because each $\max\{0,k-e+\epsilon_{i \tau(i)}\}$ is bounded by $k$.
The quasi-regularity of the system ensures that $J(\mathcal{E})\ne -\infty$ (see for instance \cite{konmipan});
then there exists an injection $\nu: \{1,\ldots,r\}\rightarrow \{1,\ldots ,n\}$
such that $\epsilon_{i,\nu(i)}\ne -\infty$ holds for all $i=1,\ldots,r$. Then
\begin{equation} \label{nu}
(k-e)r\le \sum_{i=1}^r \max\{0,k-e+\epsilon_{i \nu(i)}\}\le (k-e)r+J(\mathcal{E})
\end{equation}
holds. Since $k\ge er$, we have $k(r-1)\le (k-e)r $. Hence, from (\ref{tau}) and (\ref{nu}) we have that
$
\sum_{i=1}^r \max\{0,k-e+\epsilon_{i \tau(i)}\}\le k(r-1)\le (k-e)r\le (k-e)r+J(\mathcal{E})
$. The theorem follows.
\end{proof}\\

The following two classical examples of DAE systems show that the
bounds in Theorems \ref{indexbound} and \ref{cotaorden} may be attained.

\vskip8pt

\noindent \textbf{Example 1.\ } Let us consider a system in
Hessenberg form of size $n$ (see, for instance, \cite[Definition
2.5.3]{brenan}):

\begin{equation}\label{hessenberg}
({\Sigma})\ :=\ \left\{
\begin{array}
[c]{ccl}%
\dot{X}_1&=&f_1(X_1, \ldots , X_n) \\
&\vdots & \\ \dot{X}_i&=&f_i(X_{i-1},X_i, \ldots , X_{n-1})\ \ \  2\le i\le n-1\\ &\vdots& \\
0&=&f_n(X_{n-1})
\end{array}
\right. \end{equation}
 such that  $\left(\frac{\partial
f_n}{\partial X_{n-1}}\right)\cdot \left(\frac{\partial
f_{n-1}}{\partial X_{n-2}}\right)\cdots \left(\frac{\partial
f_2}{\partial X_{1}}\right)\cdot \left(\frac{\partial f_1}{\partial
X_{n}}\right)\ne 0$ in $\mathbb{K}$. Let us assume that the differential ideal
$\fP:=[\dot{X}_1-f_1, \ldots , \dot{X}_{n-1}-f_{n-1}, f_n]$ is
prime.

\medskip The matrix $\mathcal{E}_0$ associated with this system is
\[\mathcal{E}_0= \left(\begin{array}{ccccc}1&0&\ldots&0&0\\0&1&\ldots
&0&0\\\vdots
&\vdots&\ddots &\vdots &\vdots\\ 0&0&\ldots & 1&0\\
0&0&\ldots &0&0\end{array}\right)\] and we have the following upper
bound according to Theorem \ref{indexbound}:
\begin{equation}\label{cotahess} \sigma + {\rm ord}(\fP) \le
J(\mathcal{E}_0)+\max\{\epsilon_{ij}\} -
\min\{\epsilon_{ij}\}=n-1+1-0=n.
\end{equation}

In order to compute the differentiation index of this system,
consider the matrices
\[A:=\left(\begin{array}{ccccc}
$1$ & $0$ & \ldots &0& $0$\\$0$ & $1$ & \ldots &0&
$0$\\&\vdots&\ddots&\vdots&\\$0$ & $0$ & \ldots &1& $0$\\$0$ & $0$ &
\ldots &0& $0$\end{array}\right) \] and \[B_{\alpha,
\beta}:=\left(\begin{array}{ccccc} -\frac{\partial
f_1^{(\alpha)}}{\partial X_1^{(\beta )}} & -\frac{\partial
f_1^{(\alpha)}}{\partial X_2^{(\beta)}} & \ldots & -\frac{\partial
f_1^{(\alpha)}}{\partial X_{n-1}^{(\beta)}}&
-\frac{\partial f_1^{(\alpha)}}{\partial X_n^{(\beta )}}\\
-\frac{\partial f_2^{(\alpha)}}{\partial X_1^{(\beta )}} &
-\frac{\partial f_2^{(\alpha)}}{\partial X_2^{(\beta )}} & \ldots
&-\frac{\partial f_2^{(\alpha)}}{\partial X_{n-1}^{(\beta )}}& 0\\0
& -\frac{\partial f_3^{(\alpha)}}{\partial X_2^{(\beta )}} & \ldots
& -\frac{\partial f_3^{(\alpha)}}{\partial
X_{n-1}^{(\beta )}}&0\\
&\vdots&\ddots&\vdots&\\0 & 0 & \ldots & -\frac{\partial
f_n^{(\alpha)}}{\partial X_{n-1}^{(\beta )}}& 0\end{array}\right)
{\rm{for}} \ \alpha, \beta \ge 1.\]

Then, the corresponding matrices $\mathfrak{J}_{k,0}$ (see
Definition \ref{defimatricespi}) are

\[\left(
\begin{tabular}
[c]{ccccc}%
$A$ & \multicolumn{1}{|c}{$0$} & \multicolumn{1}{|c}{$0$}
&\multicolumn{1}{|c}{$\ldots$} & \multicolumn{1}{|c}{$0$}
\\\cline{1-1}\cline{1-1}
\\[-8mm]
&& \multicolumn{1}{|c}{} & \multicolumn{1}{|c}{}&
\multicolumn{1}{|c}{}
 \\
$B_{1, 1}$ & $A$ & \multicolumn{1}{|c}{$0$}
&\multicolumn{1}{|c}{$\ldots$} &
\multicolumn{1}{|c}{$0$} \\
\cline{1-2}\cline{2-2}
\\[-8mm]
& &  & \multicolumn{1}{|c}{}& \multicolumn{1}{|c}{}
 \\ %
$B_{2, 1}$ & $B_{2, 2}$ &$A$ &\multicolumn{1}{|c}{$\ldots$} &
\multicolumn{1}{|c}{$0$} \\
\cline{1-3}\cline{3-3}
\\[-8mm]
& &  & & \multicolumn{1}{|c}{}
 \\ %
\vdots & \vdots & \vdots &$\ddots$ &
\multicolumn{1}{|c}{$\vdots$} \\
\cline{1-4}\cline{4-4}
\\[-8mm]
& &  & &
 \\ %
$B_{(k-1), 1}$& $B_{(k-1), 2}$ & $B_{(k-1), 3}$ &$\ldots$ &$A$
\\
\end{tabular}
\right).
\]

It can be shown that the dimension  of the corresponding kernels of
$\mathfrak{J}^t_{k,0}$ are $\mu_k=k$, for $k=0,\ldots , n$.
Therefore, the differentiation index of the system $(\Sigma)$ is
$\sigma=n$ and the order of the differential ideal $\fP$ is ${\rm
ord}(\fP) = 1\cdot n - \mu_n = n-n = 0$. Thus, the upper bound
(\ref{cotahess}) is attained.

In order to verify the upper bound stated in Theorem \ref{cotaorden}
we consider the matrix

\[\mathcal{E}= \left(\begin{array}{ccccccc}1&\star&\star&\ldots&\star&\star&0\\0&1&\star&\ldots
&\star&\star&-\infty\\-\infty&0&1&\ldots &\star&\star&-\infty
\\\vdots &\vdots
&\vdots&\ddots &\vdots &\vdots &\vdots\\-\infty& -\infty&-\infty&\ldots & 1&\star&-\infty\\-\infty& -\infty&-\infty&\ldots & 0&1&-\infty\\
-\infty&-\infty&-\infty&\ldots &-\infty&0&-\infty\end{array}\right)\]
where each $\star$ stands for  $0$ or $-\infty$.

It is easy to see that the Jacobi number of this matrix is $J(\mathcal{E})=0$, which
agrees with the order of $\fP$.

\bigskip

\noindent \textbf{Example 2.\ } Consider the DAE system arising from
a variational problem describing the motion of a pendulum of length
$L$. If $g$ is the gravitational constant and $\lambda$ the force in
the bar:
\begin{equation}\label{pendulo}
({\Sigma})\ :=\ \left\{
\begin{array}
[c]{ccl}%
X_1^{(2)}-\lambda X_1&=& 0 \\
X_2^{(2)}-\lambda X_2+g&=& 0\\
X_1^2+X_2^2-L^2 &=& 0 \\
\end{array}
\right. .
\end{equation}

Let $\fP:=[X_1^{(2)}-\lambda X_1,  X_2^{(2)}-\lambda
X_2+g,X_1^2+X_2^2-L^2] \subset \mathbb{R}\{X_1, X_2, \lambda\}$.

The matrix $\mathcal{E}_0$ associated to the system is
$\left(\begin{matrix}2&0&0\\0&2&0\\0&0&0\end{matrix}\right)$ and the
bound given by Theorem \ref{indexbound} is $\sigma + {\rm
ord}(\fP)\le J(\mathcal{E}_0)+\max\{\epsilon_{ij}\} -
\min\{\epsilon_{ij}\}=4+2-0=6$.

The corresponding matrices $\mathfrak{J}_{k,1}$,
$k=1,\dots, n$ (see Definition \ref{defimatricespi}) are

{\small
\[\left(
\begin{tabular}
[c]{ccccccccccccccc}%
$1$ & $0$ & $0$ & \multicolumn{1}{|c}{$0$} & $0$ & $0$ &
\multicolumn{1}{|c}{$0$} & $0$ & $0$ & \multicolumn{1}{|c}{$0$} &
$0$ & $0$ & \multicolumn{1}{|c}{$0$} & $0$ & $0$\\ $0$ & $1$ & $0$ &
\multicolumn{1}{|c}{$0$} & $0$ & $0$ & \multicolumn{1}{|c}{$0$} &
$0$ & $0$ & \multicolumn{1}{|c}{$0$} & $0$ & $0$ &
\multicolumn{1}{|c}{$0$} & $0$ & $0$\\ $0$ & $0$ & $0$ &
\multicolumn{1}{|c}{$0$} & $0$ & $0$ & \multicolumn{1}{|c}{$0$} &
$0$ & $0$ & \multicolumn{1}{|c}{$0$} & $0$ & $0$ &
\multicolumn{1}{|c}{$0$} & $0$ & $0$\\ \cline{1-3}\cline{3-3}
\\[-8mm]
&&&&&&\multicolumn{1}{|c}{} &&& \multicolumn{1}{|c}{} &&& \multicolumn{1}{|c}{} \\ %
$0$ & $0$ & $0$ & $1$ & $0$ & $0$ & \multicolumn{1}{|c}{$0$} & $0$ &
$0$ & \multicolumn{1}{|c}{$0$} & $0$ & $0$ &
\multicolumn{1}{|c}{$0$} & $0$ & $0$\\ $0$ & $0$ & $0$ & $0$ & $1$ &
$0$ & \multicolumn{1}{|c}{$0$} & $0$ & $0$ &
\multicolumn{1}{|c}{$0$} & $0$ & $0$ & \multicolumn{1}{|c}{$0$} & $0$ & $0$\\
$0$ & $0$ & $0$ & $0$ & $0$ & $0$ & \multicolumn{1}{|c}{$0$} & $0$ &
$0$ & \multicolumn{1}{|c}{$0$} & $0$ & $0$ &
\multicolumn{1}{|c}{$0$} & $0$ & $0$\\\cline{1-6} %
\\[-8mm]
&&&&&& &&& \multicolumn{1}{|c}{} &&& \multicolumn{1}{|c}{} \\ %
$-\lambda$ & $0$ & $-X_{1}$ & $0$ & $0$ & $0$ & $1$ & $0$ & $0$ &
\multicolumn{1}{|c}{$0$} & $0$ & $0$ & \multicolumn{1}{|c}{$0$} & $0$ & $0$\\
$0$ & $-\lambda$ & $-X_{2}$ & $0$ & $0$ & $0$ & $0$ & $1$ & $0$ &
\multicolumn{1}{|c}{$0$} & $0$ & $0$ & \multicolumn{1}{|c}{$0$} & $0$ & $0$\\
$2X_{1}$ & $2X_{2}$ & $0$ & $0$ & $0$ & $0$ & $0$ & $0$ & $0$ &
\multicolumn{1}{|c}{$0$} & $0$ & $0$ & \multicolumn{1}{|c}{$0$} &
$0$ &
$0$\\\cline{1-9}%
\\[-8mm]
&&&&&& &&&  &&& \multicolumn{1}{|c}{} \\ %
$-3\dot{\lambda}$ & $0$ & $-3\dot{X}_{1}$ & $-\lambda$ & $0$ &
$-X_{1}$ & $0$ &
$0$ & $0$ & $1$ & $0$ & $0$ & \multicolumn{1}{|c}{$0$} & $0$ & $0$\\
$0$ & $-3\dot{\lambda}$ & $-3\dot{X}_{2}$ & $0$ & $-\lambda$ &
$-X_{2}$ & $0$ &
$0$ & $0$ & $0$ & $1$ & $0$ & \multicolumn{1}{|c}{$0$} & $0$ & $0$\\
$6\dot{X}_{1}$ & $6\dot{X}_{2}$ & $0$ & $2X_{1}$ & $2X_{2}$ & $0$ &
$0$ & $0$ & $0$ & $0$ & $0$ & $0$ & \multicolumn{1}{|c}{$0$} & $0$ &
$0$\\\cline{1-12}%
\\[-8mm]
&&&&&& &&&  &&& \\ %
$-6\lambda^{(2)}$ & $0$ & $-6X_{1}^{(2)}$ & $-4\dot{\lambda}$ & $0$
& $-4\dot{X}_{1}$ & $-\lambda$ & $0$ & $-X_{1}$ & $0$ & $0$ & $0$ &
$1$ & $0$ &
$0$\\
$0$ & $-6\lambda^{(2)}$ & $-6X_{2}^{(2)}$ & $0$ & $-4\dot{\lambda}$
& $-4\dot{X}_{2}$ & $0$ & $-\lambda$ & $-X_{2}$ & $0$ & $0$ & $0$ &
$0$ & $1$ &
$0$\\
$12X_{1}^{(2)}$ & $12X_{2}^{(2)}$ & $0$ & $8\dot{X}_{1}$ &
$8\dot{X}_{2}$ & $0$ & $2X_{1}$ & $2X_{2}$ & $0$ & $0$ & $0$ & $0$ & $0$ & $0$ & $0$%
\end{tabular}
\right).
\]
}

The dimension  of the corresponding kernels of
$\mathfrak{J}^t_{k,1}$ are $\mu_k=k$, for $k=0,\ldots ,4$, and
$\mu_5=4$. Then, the differentiation index of the system $(\Sigma)$
is $\sigma=4$. We remark that the pendulum system is
customarily considered as a system of differentiation index $3$, but
in fact this is the value of the differentiation index of its
first-order equivalent system; as we observed in \cite[Subsection
5.1]{DJS07} differential changes of coordinates may alter the
differentiation index.

On the other hand, the order of the
associated differential ideal is ${\rm ord}(\fP) = er - \mu_\sigma =
2\cdot 3-4=2$ (see Proposition \ref{ordenmu}).

Therefore, we have that $\sigma +{\rm ord}(\fP)  = 4+2 =6$, which
coincides with our stated upper bound.

Finally we have $\mathcal{E}=\left(\begin{matrix}2&-\infty
&0\\-\infty &2&0\\0&0&-\infty
\end{matrix}\right)$. The Jacobi number of this matrix equals $2$ and thus, the Jacobi
bound in Theorem \ref{cotaorden} coincides with the order of the ideal.

\section{The membership problem for quasi-regular differential
ideals} \label{membership}

Roughly speaking, the \emph{ideal membership problem} in an
arbitrary commutative ring $A$ consists in deciding if a given
element $f\in A$ belongs to a fixed ideal $I\subset A$, and, in the
affirmative case, representing $f$ as a polynomial linear
combination of a given set of generators of $I$.

Since the work of G. Hermann \cite{herrmann}, it is well known that
the membership problem is \emph{decidable} for polynomial rings in
finitely many variables. On the contrary, this is not the case in
the differential context, where the problem is \emph{undecidable}
for arbitrary ideals (see \cite{gallo}) and remains still an open
question for finitely generated ideals. However, there are special
classes of differential ideals for which the problem is decidable,
in particular the class of differential radical ideals
(\cite{Seidenberg}, see also \cite{Boulier}).

Concerning the representation problem, besides the non
noetherianity, the differential case involves another additional
ingredient: the order $N$ of derivation of the given generators of
$I$ needed to write an element $f\in I$ as a polynomial linear
combination of the generators and their first $N$ total derivatives.
The known order bounds seem to be too big, even for radical ideals
(see for instance \cite{golubitsky}, where an upper bound in terms
of the Ackerman function is given). Obviously, once the order is
bounded, the problem becomes a purely algebraic representation
problem instance in a suitable polynomial ring in finitely many
variables.

By means of Theorem \ref{sadik} above, we are able to give
\emph{efficient} order bounds for the membership problem in the
quasi-regular differential setting which lead to estimations for the
degrees and number of variables involved in the representation.
The quasi-regularity condition also ensures that the
(algebraic) ideals involved are prime and generated by regular
sequences, which allows a significant improvement in the degree
bounds for the representation with respect to more general
situations.\\

We will use the following membership theorem for polynomial rings:

\begin{theorem} \label{fitchas}
(\cite[Theorem 5.1]{dickenstein})
Let $k$ be a field and $g_1,\ldots, g_s\in k[Z_1,\ldots,Z_n]$ be a complete intersection of polynomials
whose total degrees are bounded by an integer $d$.
Let $g\in k[Z_1,\ldots,Z_n]$ be another polynomial.
Then the following conditions are equivalent:
\begin{enumerate}
\item $g$ belongs to the ideal generated by $g_1,\ldots,g_s$;
\item there exist polynomials $a_1,\ldots,a_s$ such that $g=\sum a_jg_j$ and
$\deg(a_jg_j)\le d^s+\deg(g)$ for $1\le j\le s$. $\blacksquare$
\end{enumerate}
\end{theorem}

With the same notations as above we have the following effective
differential membership:

\begin{theorem} \label{member bound}
Let $(\Sigma)$ be a quasi-regular differential system in the sense
of Remark \ref{primo} and Section \ref{seccionjacobiana},
defined by polynomials $F:=f_1,\dots, f_r\in
K\{X\}$, and let $D$ be an upper bound for the total degrees of
$f_1,\dots, f_r$. Let $f\in K\{X\}$ be an arbitrary differential
polynomial in the differential ideal $[F]$. Set $N:=\sigma +
\max\{-1, \, {\rm ord}(f)-e \}$, where $\sigma$ is the
differentiation index of $(\Sigma)$. Then, a representation
\[f\ =\ \sum_{\genfrac{}{}{0pt}{}{1\le i\le r}{0\le j\le N}} g_{ij}\, f_i^{(j)}\]
holds in the ring $A_{N+e}$ where each polynomial $g_{ij}$ has total
degree bounded by $\deg(f)+D^{r(N+1)}$.
\end{theorem}

\begin{proof}
The upper bound on the order of derivation of the polynomials
$f_1,\dots, f_r$ is a direct consequence of Theorem \ref{sadik}
applied to $i:=\max\{e-1,\, \textrm{ord}(f)\}$. The degree upper
bound for the polynomials $g_{ij}$ follows from Proposition
\ref{regular sequence en P} and Theorem \ref{fitchas}.
\end{proof}

\begin{remark} From Theorems \ref{sadik} and \ref{indexbound} we deduce that
for every $i\in \N_{\ge e-1}$, the equality
$\Delta_{i+1+J(\mathcal{E}_0) - \min\{\epsilon_{ij}\}}\cap B_i =
\Delta \cap B_i$ holds, which provides a completely syntactical
upper bound for the derivation orders and degrees in the membership
problem: it suffices to take $N:=J(\mathcal{E}_0) -
\min\{\epsilon_{ij}\} +\max \{{\rm ord}(f), e-1\}$.
\end{remark}

\section{DAE and ODE systems}\label{section ode}

\subsection{The $\fP$-differentiation index and an explicit ODE system}
\label{ODEdim0}

In the zero-dimensional case ($n=r$), the estimation for the
Hilbert-Kolchin regularity of the ideal $\fP$ allows us to give a
result concerning the number of derivatives of the input equations
required to obtain an explicit ODE system from the system
(\ref{sistema ampliado}). We will show that this number is at most
the $\fP$-differentiation index of the system:

\begin{theorem} \label{despeje general dimension 0} Let $(\Sigma)$ be a DAE system
as in (\ref{sistema ampliado}) of differential dimension $0$ (or
equivalently, $r=n$), maximal order bounded by $e$ and
$\fP$-differentiation index $\sigma$. Let $\Xi =\{
\xi_1^{(\ell_1)},\ldots,\xi_s^{(\ell_s)}\}\subseteq \{ X^{[e-1]}\}$
be an algebraic transcendence basis of the fraction field of
$B_{e-1}/\Delta \cap B_{e-1}$ over the field $K$. Then:
\begin{enumerate}

\item for each $i=1,\ldots ,s$ there exists a non-zero separable polynomial
$P_i$ with coefficients in the base field $k$, such that $P_i(\Xi,
\xi_i^{(e)})\in (f_1^{[\sigma]},\ldots,  f_r^{[\sigma]})\subset
B_{e+\sigma}$;

\item  set $\{\eta_{s+1}, \ldots , \eta_{n}\}:=\{X\}\setminus \{\xi_1,
\ldots , \xi_s\}$. Then, 
for all $i=s+1,\ldots ,n$, there exists a non-zero separable polynomial $P_i$
with coefficients in the base field $K$, such that $P_i(\Xi, \eta_i^{(e-1)})\in
(f_1^{[\sigma-1]},\ldots, f_r^{[\sigma-1]})\subset B_{e+\sigma-1}$.
\end{enumerate}

In particular, for every $i=1, \ldots , n$  there exists a separable
non-trivial polynomial relation between $X_i^{(e)}$ and $\Xi$ modulo $\Delta$
which can be obtained using at most $\sigma$ derivations of the input
equations.
\end{theorem}

\begin{proof}  Since we are in a differential zero-dimensional situation, from the
upper bound on the regularity of the Hilbert-Kolchin function (Theorem
\ref{hilbert constante}), we have that the set $\Xi$ is also an algebraic
transcendence basis of the fraction field of the ring $B_{e}/\Delta\cap B_{e}$.
Then, for $i=1, \ldots , s$,  there exists a polynomial $P_i$ in $s+1$
variables with coefficients in $K$, such that $P_i(\Xi,\xi_{i}^{(e)})$ belongs
to the ideal $\Delta\cap B_e=\Delta_{\sigma+1}\cap B_e$ (Theorem \ref{sadik}).
Clearly, this polynomial can be chosen  separable.

The second assertion follows similarly, but in this case we use the fact that
the family  $\{\eta_i^{(e-1)}, \Xi\}$ is algebraically dependent when regarded
in the fraction field of $B_{e-1}/\Delta\cap B_{e-1}$ over $K$, for all $i=s+1,
\ldots , n$.
\end{proof}

\subsection{Existence and Uniqueness of Solutions}\label{exun}

In the case of (explicit) ODE systems, it is well known and
classical that, under certain general hypotheses, one can ensure the
existence  and uniqueness of solutions when certain initial
conditions are fixed. However in the case of \textit{implicit} DAE
systems the same problem has only been considered recently (see for
instance \cite{pritchard}, \cite{sit}, \cite{rabier}).

This last subsection is devoted to the problem of the existence and
uniqueness of solutions for an ubiquitous class of zero-dimensional
implicit autonomous DAE systems. More precisely, throughout this
subsection we consider DAE systems
\begin{equation} \label{sistemade1}
(\mathfrak{S}):=\left\{
\begin{array}
[c]{ccl}%
f_1(X,\dot{X}) &=& 0 \\
&\vdots &\\
f_n(X,\dot{X}) &=& 0 \\
\end{array}
\right. ,
\end{equation}
where $f_1,\dots, f_n$ are polynomials in the $n$ differential
unknowns $X:=X_1,\ldots, X_n$ and  their first derivatives, with
coefficients in  $\mathbb{C}$. We also assume  that the equations
generate a prime zero-dimensional differential ideal
$\mathfrak{Q}:=[f_1, \ldots ,f_n]\subset \C\{X\}$ such that the
system is quasi-regular at $\mathfrak{Q}$ and satisfies the hypothesis
of Section \ref{seccionjacobiana}. We denote by $\sigma$ the
differentiation index of the system $(\mathfrak{S})$.

Under these assumptions (cf. Remark \ref{primo} and Proposition
\ref{regular sequence en P}), the ideal $\mathfrak{Q}_{\sigma+1}$
generated by the first $\sigma$ total derivatives of the polynomials
$f_i$ is a prime ideal defining an algebraic variety $W$ contained
in $\mathbb{C}^{n(\sigma+2)}$.

Set $V_1\subset \mathbb{C}^{2n}$ and $V_0\subset\mathbb{C}^{n}$ for
the projection of $W$ to the coordinates $(x,\dot{x})$ and $x$,
respectively. So, $V_1$ and $V_0$ are defined by the contraction of
the prolonged prime ideal $\mathfrak{Q}_{\sigma+1}$ to the
corresponding polynomial rings $\mathbb{C}[X,\dot{X}]$ and
$\mathbb{C}[X]$ respectively. Let us remark that from Theorem
\ref{sadik}, these varieties are also defined by the contractions of
the ideal $\mathfrak{Q}\subset \mathbb{C}\{X\}$ to the corresponding
rings.

Roughly speaking our existence and uniqueness theorem will state
that for any point $(x_0,p_0)$ in a suitable Zariski open dense
${U}\subseteq V_1$ the system admits a unique solution
$\varphi:(-\varepsilon,\varepsilon)\rightarrow \mathbb{C}^n$ such
that $\varphi(0)=x_0$ and $\dot{\varphi}(0)=p_0$. As customarily
(see for instance \cite[Section 6]{rabier}), the determination of
the open subset ${U}$ needs additional properties concerning the
smoothness of both the initial point $(x_0,p_0)$ and the related
projection $\pi:V_1\rightarrow V_0$, $(x,\dot{x})\mapsto x$.

In order to formalize these ideas we recall the notion of
unramifiedness (see for instance \cite[\S 6, page 100]{kunz1}).

\begin{definition} Let $S$ be an $R$-algebra, $\mathfrak{p}\subset
S$ a prime ideal and $\mathfrak{q}\subset R$ its contraction to $R$.
The ideal $\mathfrak{p}$ is \emph{unramified} if
$\mathfrak{q}S_{\mathfrak{p}}=\mathfrak{p}S_{\mathfrak{p}}$ holds
and the field extension of the residual fields
$k(\mathfrak{p})|k(\mathfrak{q})$ is separable algebraic.

If $F:V\rightarrow W$ is a morphism of complex algebraic affine
varieties and $p\in V$, we say that \emph{$F$ is unramified at $p\in
V$} if the maximal ideal associated to $p$ is an unramified prime
ideal of $\mathbb{C}[W]$ for the structure of
$\mathbb{C}[V]$-algebra induced by $F$ (observe that in this case
the separability condition is automatically fulfilled).
\end{definition}

We recall that the unramified points form an open Zariski dense
subset of the source space (see for instance \cite[Corollary
6.10]{kunz1}).

Let us reinterpret the condition of unramifiedness for our
projection $\pi: V_1\rightarrow V_0$ at a point $(x_0,p_0)\in V_1$:
if $\mathfrak{M}\subset \mathbb{C}[V_1]$ denotes the maximal ideal
associated to $(x_0,p_0)$, then the contracted ideal
$\mathfrak{M}\cap\mathbb{C}[V_0]$ agrees with the maximal ideal
$\mathfrak{N}\subset \mathbb{C}[V_0]$ associated to the point
$\pi(x_0,p_0)=x_0\in V_0$. So, the unramifiedness says that
$\mathfrak{N}$ must generate the maximal (hence regular) ideal of
the local ring $\mathbb{C}[V_1]_{\mathfrak{M}}$. Equivalently, we
have the equality ${\mathfrak{M}}= I(V_1)+\mathfrak{N}$ in the local
ring $\mathbb{C}[X,\dot{X}]_{\mathfrak{M}}$ (here $I(V_1)$ denotes
the ideal of the variety $V_1$).

Therefore, if $H$ is any system of generators of $I(V_1)$, the
polynomials $H(x_0,\dot{X})$ generate the maximal ideal of $p_0$ in
the polynomial ring $\mathbb{C}[\dot{X}]$ localized at the maximal
ideal corresponding to $p_0$. Hence, in terms of the Jacobian matrix
$DH(x_0,p_0)$, the unramifiedness of $\pi$ at the point $(x_0,p_0)$
is equivalent to the fact that the submatrix of $DH(x_0,p_0)$
corresponding to those derivatives with respect to the variables
$\dot{X}$ has full column rank $n$.

Let us remark that under this last form the unramifiedness condition
becomes very similar to the conditions induced by the hypotheses of
``\emph{non-singularity}" and ``\textit{reducible $\pi$-manifold}"
given in Rabier \& Rheinboldt's paper (see \cite[Theorem 5.2 and
Definition 6.1]{rabier}) required in order to prove their existence
and uniqueness result \cite[Theorem 6.1]{rabier}. On the other hand,
our requirement is less restrictive than the birationality condition
assumed in \cite[Theorem 25]{pritchard}.

Now we are able to show the main result of this subsection:

\begin{theorem}\label{exyun} Let $(\mathfrak{S})$ be a  DAE system
as in (\ref{sistemade1}). Let $\mathfrak{Q}=[f_1, \ldots
,f_n]\subset \C\{X\}$ be the associated differential ideal and satisfies
the hypothesis of Section \ref{seccionjacobiana}. Assume
that $\mathfrak{Q}$ is prime and $(\mathfrak{S})$ is quasi-regular.
Let $V_0\subset \mathbb{C}^n$ and $V_1\subset \mathbb{C}^{2n}$ be
the irreducible varieties defined by the ideals $\mathfrak{Q}\cap
\mathbb{C}[X]$ and $\mathfrak{Q}\cap \mathbb{C}[X,\dot{X}]$
respectively and let $\pi: V_1\rightarrow V_0$ be the projection to
the first $n$ coordinates.

Then, for every regular point $(x_0, p_0) \in V_1$ where $\pi$ is
unramified, there exist $\varepsilon
>0$, a relative open neighborhood $\mathcal{U}\subset V_1$ of $(x_0, p_0)$
and a unique analytic function $\varphi:(-\varepsilon ,
\varepsilon)\to \C^n$ such that $(\varphi ,
\dot{\varphi}):(-\varepsilon , \varepsilon)\to \mathcal{U}$,
$\varphi(0)=x_0$  and for every $i=1, \ldots ,n $, $f_i(\varphi(t),
\dot{\varphi}(t))=0$ for $t\in (-\varepsilon, \varepsilon)$.
\end{theorem}

\begin{proof}
Let $G\subset \mathbb{C}[X]$ and $H\subset \mathbb{C}[X,\dot{X}]$ be
systems of generators of $\mathfrak{Q}\cap \mathbb{C}[X]$ and
$\mathfrak{Q}\cap \mathbb{C}[X,\dot{X}]$ respectively. Without loss
of generality, we may assume that $G\subset H$.

{}From the unramifiedness condition the submatrix $D_{\dot X} H$
consisting of last $n$ columns of the Jacobian matrix $DH$ has full
column rank $n$ at $(x_0, p_0)$. In addition, since $(x_0, p_0)$ is
a regular point of $V_1$, the rank of the matrix $DH(x_0, p_0)$
equals $2n - d$, which is the codimension of $V_1$. Then, without
loss of generality (after a permutation of the variables $X$), we
may assume that the submatrix of $DH(x_0, p_0)$ consisting of its
last $2n-d$ columns has full column rank.

Let us write $\textbf{x} := X_1,\dots, X_d$ and $\textbf{y}:=
X_{d+1}, \ldots , X_n$ for the variables $X$, and
$x_0:=(\textbf{x}_0, \textbf{y}_0)\in \C^d\times \C^{n-d}$. By the
Implicit Function Theorem (IFT),  there exist open neighborhoods
$Z_0\subset \C^d$ of
 ${\textbf x}_0$ and $Z_1\subset \C^{2n-d}$ of
 $({\textbf y}_0, p_0)$
 and holomorphic functions $(h, \xi, \eta): Z_0\to Z_1$ such that the open subset of $V_1$ given by
  $\{H=0\}\cap (Z_0\times Z_1)$ can be described by the new, explicit system
  \begin{equation}\label{sistemadespe}\left\{\begin{array}{ccc}\textbf y &=&h(\textbf x )\\ \dot{\textbf x
}&=&\xi(\textbf x )\\ \dot{\textbf y }&=&\eta(\textbf x
)\end{array}\right..\end{equation}

Note that if $H = \{G, \widetilde G\}$, the Jacobian matrix of $H$
at $(x_0, p_0)$ takes the form
$$DH(x_0, p_0) = \left(\begin{array}{c|c}
\begin{array}{cc}
D_\textbf{x}G (x_0) & D_\textbf{y}G(x_0)
\end{array} & \quad  0 \quad \\
\hline\\[-3mm] \begin{array}{cc}
D_\textbf{x}\widetilde G(x_0, p_0) & D_\textbf{y}\widetilde G(x_0,
p_0)
\end{array} & \quad D_{\dot X} \widetilde G (x_0, p_0) \quad
\end{array}\right).$$
Then, we have that $\text{rank}(D_\textbf{y}G(x_0)) = n-d$ and so,
by the IFT applied to the system $G=0$ at the point $x_0$, the
system $\textbf{y} =h(\textbf x)$ describes a neighborhood of $x_0$
in $V_0$.

Let us consider now the system  $\dot{\textbf x }=\xi(\textbf x )$
as an autonomous ODE. By the classical Existence Theorem (see for
instance \cite[Theorem 1, \S 3, Chapter 8]{hirsch}), there exist
$\varepsilon
>0$ and a unique analytic function $\psi:=(\psi_1, \ldots ,
\psi_d): (-\varepsilon,\varepsilon)\rightarrow \mathbb{C}^d$, such
that $\psi(0)=\textbf x _0$  and, for $j=1, \ldots d$,
$\frac{d}{dt}\psi_j=\xi_j\circ\psi$. Thus, the  function
$\varphi:(-\varepsilon,\varepsilon)\rightarrow \mathbb{C}^n$ defined
by $\varphi(t)=(\psi(t), h(\psi(t))$ is a solution of the mixed
differential system $\left\{\begin{array}{ccc}\textbf y &=&h(\textbf
x )\\ \dot{\textbf
x }&=&\xi(\textbf x )\end{array}\right.$. 

It remains to verify that  $\varphi$ is a solution of the original
system $(\mathfrak{S})$. Since $(f_1, \ldots , f_n)\subset
\mathfrak{Q}\cap A_1 =(H)$,  it is enough to check that $\varphi$
verifies  the system (\ref{sistemadespe}) since it defines an open
subset of $V_1=\{H=0\}$ containing the point $(x_0, p_0)$. By
construction, one only needs to check that
$\frac{d}{dt}\,\varphi_{d+j} =\eta_j \circ \varphi$, for all
$j=1,\ldots , n-d$, or, more explicitly, that $\frac{d}{dt}\,
(h_j\circ \psi) =\eta_j \circ \psi$.

For every $g\in \mathfrak{Q}\cap A_0$,  we  know that
$g(\textbf{x},h(\textbf{x}))=0$ in a neighborhood of $\textbf{x}_0$.
Taking the partial derivative with respect to the variable $X_i$,
for $i=1,\ldots ,d$, we have that
$ \frac{\partial g}{\partial X_i}(\textbf{x},h(\textbf{x}))+
\sum_{j=1}^{n-d} \frac{\partial g}{\partial
X_{d+j}}(\textbf{x},h(\textbf{x}))\ \frac{\partial h_j}{\partial
X_i}(\textbf{x}) =0 $
for $\textbf{x}$ in a neighborhood of $\textbf{x}_0$. Multiplying
this identity by $\xi_i(\textbf{x})$ and adding over all the indices
$i$, we obtain:
\begin{equation} \label{11}
\sum_{i=1}^d\frac{\partial g}{\partial
X_i}(\textbf{x},h(\textbf{x}))\ \xi_i(\textbf{x}) + \sum_{j=1}^{n-d}
\frac{\partial g}{\partial X_{d+j}}(\textbf{x},h(\textbf{x})) \ \sum
_{i=1}^d \frac{\partial h_j}{\partial X_i}(\textbf{x}) \
\xi_i(\textbf{x})=0 .
\end{equation}
On the other hand, since the total derivative of  $g$ belongs to the
ideal $\mathfrak{Q}\cap A_1$, we have that
$\dot{g}(\textbf{x},h(\textbf{x}),\xi(\textbf{x}),\eta(\textbf{x}))=0$
for $\textbf{x}$ in a neighborhood of  $\textbf{x}_0$, that is,
\begin{equation} \label{22}
\sum_{i=1}^{d} \frac{\partial g}{\partial X_i}
(\textbf{x},h(\textbf{x})) \ \xi_i(\textbf{x}) + \sum_{j=1}^{n-d}
\frac{\partial g}{\partial X_{d+j}} (\textbf{x},h(\textbf{x}))\
\eta_j(\textbf{x}) =0.
\end{equation}
{}From identities (\ref{11}) and (\ref{22}) we deduce that, in a
neighborhood of $\textbf{x}_0$,
\[
\sum_{j=1}^{n-d} \frac{\partial g}{\partial
X_{d+j}}(\textbf{x},h(\textbf{x})) \Big(\eta_j(\textbf{x})- \sum
_{i=1}^d \ \frac{\partial h_j}{\partial X_i}(\textbf{x}) \
\xi_i(\textbf{x})\Big) =0.
\]

In particular, by using this identity for all  $g\in G$, we conclude
that, for all $\textbf{x}$ in a neighborhood of $\textbf{x}_0$, the
vector $w(\textbf{x})\in\mathbb{C}^{n-d}$, whose $j$-th coordinate
is
$\eta_j(\textbf{x})- \sum _{i=1}^d  \ \frac{\partial h_j}{\partial
X_i}(\textbf{x}) \ \xi_i(\textbf{x})$
for $j=1,\ldots,n-d$, belongs to the kernel of the submatrix
$D_\textbf{y} G(\textbf{x},h(\textbf{x}))$. Since this matrix has
full column rank $n-d$, it follows that $w(\textbf{x})=0$ and thus
\[ \hspace{4.5cm}
\eta_j(\textbf{x})\ =\ \sum _{i=1}^d  \ \frac{\partial h_j}{\partial
X_i}(\textbf{x}) \ \xi_i(\textbf{x}) \hspace{2cm} j=1,\dots, n-d.
\]
Setting  $\textbf{x}=\psi(t)$ we conclude that
\[\eta_j \circ
\psi = \sum _{i=1}^d  \ \frac{\partial h_j}{\partial X_i}(\psi(t)) \
\xi_i(\psi(t)) = \sum _{i=1}^d  \ \frac{\partial h_j}{\partial
X_i}(\psi(t)) \ \dot{\psi_i}(t)= \frac{d}{dt}(h_j\circ \psi).
\]

The uniqueness of the solution is a straightforward consequence of
the classical theorem for ODE's.
\end{proof}

\bigskip

We remark that the previous existence and uniqueness result can be easily
extended to non autonomous systems in, at least, two well known ways.

First, suppose that the base field
$K$ in the square system (\ref{sistemade1}) is the field of rational complex
functions $\mathbb{C}(t)$ instead of $\mathbb{C}$. Let $r(t)\in \mathbb{C}(t)$
be the maximum common divisor
of the denominators appearing in (\ref{sistemade1}) and consider the
new system

\begin{equation} \label{sistemade2}
(\mathfrak{S}_1):=\left\{
\begin{array}
[c]{ccl}%
r(t)f_1(X,\dot{X}) &=& 0 \\
&\vdots &\\
r(t)f_n(X,\dot{X}) &=& 0 \\
\end{array}
\right. ,
\end{equation}

\noindent which involves polynomials in $\mathbb{C}[t,X,\dot{X}]$.

This system can be interpreted as an autonomous one by simply considering $t$ as a new
unknown variable and adding the equation
\[
\dot{t}=1.
\]

It is easy to see that this new system is quasi-regular, the involved equations generate a
differential ideal in $\mathbb{C}\{t,X\}$ which is prime after inverting the polynomial
$r(t)$ (in fact the corresponding localization of the factor ring can be naturally included in
the factor ring of the original system which is an integral domain by hypothesis). Also the
system verifies the hypothesis of Subsection \ref{seccionjacobiana} as the original one.

Moreover, the localization in $r(t)$ by the Rabinowitz's Trick can be replaced by introducing
a new variable $S$ and the (polynomial) equation
\[
S\, r(t)-1=0.
\]

Now, we can applied the previous result to this $(n+2)\times(n+2)$-autonomous system and
extend Theorem \ref{exyun} for the non-autonomous case.\\

On the other hand, as it is observed in \cite[Section 2.2]{sit}, we may also extend our results
to non-autonomous systems involving transcendental functions which may be defined as solutions of
single DAE-equations (for instance, trigonometric or exponential functions):
one replaces the transcendental function by a new variable and adds the autonomous
differential equation which defines that function. Since each added equation involves
a new variable the assumptions about the original system will be fulfilled also by the new one.

\bigskip

\noindent \textbf{Acknowledgements.} The authors are grateful to
Fran\c cois Ollivier (\'Ecole Polytechnique, Palaiseau) for helpful
remarks on Jacobi's work and for providing them his translations of
the papers \cite{jacobi}, \cite{jacobi2}, and \cite{kondra}.
Gabriela Jeronimo and Gustavo Massaccesi also thank Flavia Bonomo
for bibliographic information on the K\"onig-Egerv\'ary theorem.
\par Part of this work was carried out while
Gabriela Jeronimo was visiting the Institute for Mathematics and its
Applications (IMA), Minneapolis MN, USA. She thanks the IMA for its
hospitality.

\end{document}